\documentclass[review]{elsarticle}
\makeatletter
\def\ps@pprintTitle{%
	\let\@oddhead\@empty
	\let\@evenhead\@empty
	\def\@oddfoot{\centerline{\thepage}}%
	\let\@evenfoot\@oddfoot}
\makeatother
\usepackage{lineno,hyperref}
\modulolinenumbers[5]

\usepackage{amsmath,amssymb,epsfig,amsthm,bm}
\usepackage{a4wide,xcolor,eucal,enumerate,mathrsfs,graphics,caption,subfig,booktabs,verbatim}
\usepackage{dsfont,natbib}

\biboptions{authoryear}
\journal{}
\theoremstyle{definition}
\newtheorem{de}{Definition}[section]

\theoremstyle{plain}
\newtheorem{theo}[de]{Theorem}
\newtheorem{lemma}[de]{Lemma}

\newtheorem{cor}[de]{Corollary}
\theoremstyle{remark}
\newtheorem{re}[de]{Remark}

\newcommand{\R}{\mathbb{R}}

\newcommand{\p}{\mathbb{P}}

\newcommand{\E}{\mathbb{E}}

\newcommand{\F}{\mathcal{F}}

\newcommand{\1}{\mathds{1}}

\begin{document}

\begin{frontmatter}
	
	\title{Isotonized smooth estimators of a monotone baseline hazard in the Cox model}
	
	
	\author[mymainaddress]{Hendrik P.~Lopuha\"a}
	
	\author[mymainaddress]{Eni Musta\corref{mycorrespondingauthor}}
	\cortext[mycorrespondingauthor]{Corresponding author}
	\ead{e.musta@tudelft.nl}
	
	\address[mymainaddress]{DIAM, Faculty EEMCS, Delft University of Technology, Mekelweg 4, 2628 CD Delft, The Netherlands}

\begin{abstract}
We consider two isotonic smooth estimators for a monotone baseline hazard in the Cox model,
a maximum smooth likelihood estimator and a Grenander-type estimator based on the smoothed Breslow estimator for the
cumulative baseline hazard.
We show that they are both asymptotically normal at rate $n^{m/(2m+1)}$, where $m\geq 2$ denotes the level of smoothness considered,
and we relate their limit behavior to kernel smoothed isotonic estimators studied in~\cite{LopuhaaMustaSI2016}.
It turns out that the Grenander-type estimator is asymptotically equivalent to the kernel smoothed isotonic estimators,
while the maximum smoothed likelihood estimator exhibits the same asymptotic variance but a different bias.
Finally, we present numerical results on pointwise confidence intervals that illustrate the comparable behavior of the two methods.
\end{abstract}
\end{frontmatter}

\section{Introduction}
For studying lifetime distributions in the presence of right censored survival data,
the Cox regression model is a very popular method that allows incorporation of covariates.
The fact that the regression coefficients (parametric component) can be estimated while leaving the baseline distribution (nonparametric component) unspecified,
together with its ease of interpretation, resulting from the formulation in terms of the hazard rate,
as well as the proportional effect of the covariates,
favor the wide use of this semi-parametric model, especially in medical applications.
Since its first introduction (see~\cite{Cox72}), much effort has been spent on giving a firm mathematical basis to this approach.
Initially, the attention was on the derivation of large sample properties of the maximum partial likelihood estimator of the regression coefficients and of the Breslow estimator for the cumulative baseline hazard (e.g., see~\cite{Efron72}, \cite{Cox75}, \cite{Tsiatis81}).
Although the most attractive property of this approach is that it does not assume any fixed shape on the hazard curve, there are several cases
where order restrictions, such as monotonicity, better match the practical expectations.
An example can be found in~\cite{Geloven13}, and other references therein,
concerning a large clinical trial for patients with acute coronary syndrome that exhibit a decreasing risk pattern.
Traditional nonparametric estimators, such as the Kaplan-Meier, Nelson-Aalen, or Breslow estimator, do not incorporate
a decreasing risk pattern, and a monotone nonparametric estimate of the hazard rate is called for.
Estimation of the baseline hazard function under monotonicity constraints was first studied in~\cite{CC94} and
more recently by~\cite{LopuhaaNane2013}, who investigate the maximum likelihood estimator and a Grenander-type estimator
defined as the slope of the greatest convex minorant (or least concave majorant) of the Breslow estimator.

Traditional isotonic estimators, such as maximum likelihood estimators and Grenander-type estimators,
are step functions that exhibit a non normal limit distribution at rate~$n^{1/3}$.
On the other hand, a long stream of research has shown that,
if one is willing to assume more regularity on the function of interest,
smooth estimators can be used to achieve a faster rate of convergence to a Gaussian distributional law and to estimate derivatives.
Typically, these estimators are constructed by combining an isotonization step with a smoothing step.
Estimators constructed by smoothing followed by an isotonization step have been considered
in~\cite{chenglin1981}, \cite{wright1982}, \cite{friedmantibshirani1984}, and~\cite{ramsay1998}, for the regression setting,
in~\cite{vdvaart-vdlaan2003} for estimating a monotone density,
and in~\cite{eggermont-lariccia2000}, who consider maximum smoothed likelihood estimators for monotone densities.
Methods that interchange the smoothing step and the isotonization step,
can be found in~\cite{mukerjee1988}, \cite{DGL13}, and~\cite{LM15}, who study kernel smoothed isotonic estimators.
Comparisons between isotonized smooth estimators and smoothed isotonic estimators are made in~\cite{mammen1991} for the regression setting,
in~\cite{GJW10} for the current status model,
and in~\cite{GJ13}, who investigate a smoothed maximum likelihood estimator and a penalized least squares estimator for a monotone hazard.

In~\cite{Nane}, several smooth monotone estimators for a monotone baseline hazard in the Cox model have been introduced, which were shown to be consistent.
Two of these methods are kernel smoothed versions of the maximum likelihood estimator and the Grenander-type estimator from~\cite{LopuhaaNane2013}.
Both methods have been studied by~\cite{LopuhaaMustaSI2016} and were shown to be asymptotically normal at rate $n^{m/(2m+1)}$,
where $m$ denotes the level of smoothness of the baseline hazard.

In this paper we investigate two other estimators,
for which the order of the smoothing step and the isotonization step is interchanged.
The first estimator that we consider is the maximum smoothed likelihood estimator.
This estimator is obtained by first smoothing the loglikelihood of the Cox model
and then find the maximizer of the smoothed likelihood among all decreasing baseline hazards.
By first smoothing the loglikelihood, one avoids the discrete behavior of the traditional MLE.
This approach is similar to the methods in~\cite{eggermont-lariccia2000} for monotone densities and in~\cite{GJW10} for the current status model.
The second estimator is a Grenander-type estimator based on the smoothed Breslow estimator.
Grenander-type estimators for a nondecreasing curve are obtained as the left-derivative of the greatest convex minorant of a naive nonparametric estimator
for the integrated curve of interest, see~\cite{grenander1956} and also~\cite{durot2007} among others.
For our setup, the smoothed Breslow estimator serves as an estimator for the cumulative baseline hazard.
By smoothing the Breslow estimator, one avoids the discrete behavior of the left-derivative of its least concave majorant.
This second approach is similar to the methods considered in~\cite{chenglin1981}, \cite{wright1982}, \cite{friedmantibshirani1984}, and~\cite{vdvaart-vdlaan2003},
and to one of the two methods studied in~\cite{mammen1991}.
Asymptotic normality at rate $n^{m/(2m+1)}$ is established for both estimators, for which we rely on techniques developed in~\cite{GJW10}.
The key idea is that the isotonized smooth estimator can be represented as a least squares projection of a naive smooth estimator.
The latter estimator is not monotone, but much simpler to analyze and it is shown to be asymptotically equivalent to the smooth isotonic estimator.

The isotonized smoothed Breslow estimator is shown to be asymptotically equivalent to the smoothed Grenander-type estimator studied in~\cite{LopuhaaMustaSI2016}.
This means that the order of smoothing and isotonization is irrelevant, which is in line with the findings in~\cite{mammen1991}.
The maximum smoothed likelihood estimator exhibits the same variance as the previous ones but has a different asymptotic bias,
a phenomenon that was also encountered in~\cite{GJW10}.
A small simulation study shows that no method performs strictly better than the other.

The paper is organized as follows.
In Section~\ref{sec:model} we specify the Cox regression model and provide some background information that will be used  in the sequel.
The maximum smoothed likelihood estimator is considered in Section~\ref{sec:MSLE}
and the isotonized smoothed Breslow estimator in Section~\ref{sec:GS}.
We only consider the case of a non-decreasing baseline hazard.
The same results can be obtained similarly for a non-increasing hazard.
The results of a small simulation study are reported in Section~\ref{sec:conf-int}.
All the proofs have been put in an appendix at the end of the paper.

\section{The Cox regression model}\label{sec:model}
Let $X_1,\dots,X_n$ be an i.i.d.~sample representing the survival times of $n$ individuals, which can be observed only on time intervals $[0,C_i]$ for some i.i.d. censoring times $C_1,\dots,C_n$.
The observations consists of  i.i.d. triplets $(T_1,\Delta_1,Z_1),\dots,(T_n,\Delta_n,Z_n)$, where $T_i=\min(X_i,C_i)$ denotes the follow up time,  $\Delta_i=\1_{\{X_i\leq C_i\}}$ is the censoring indicator, and $Z_i\in\R^p$ is a time independent covariate vector.
Given the covariate vector $Z,$ the event time $X$ and the censoring time $C$ are assumed to be independent. Furthermore, conditionally on $Z=z,$ the event time is assumed to be a nonnegative random variable with an absolutely continuous distribution function~$F(x|z)$ and density $f(x|z).$
Similarly the censoring time is assumed to be a nonnegative r.v. with an absolutely continuous distribution function $G(x|z)$ and density~$g(x|z).$
The censoring mechanism is assumed to be non-informative, i.e., $F$ and $G$ share no parameters.
Within the Cox model, the conditional hazard rate $\lambda(x|z)$ for a subject with covariate vector~$z\in\R^p$, is related to the corresponding covariate by
\begin{equation}
\label{eq:cox model}
\lambda(x|z)=\lambda_0(x)\,\mathrm{e}^{\beta'_0z},\quad x\in\R^+,
\end{equation}
where $\lambda_0$ represents the baseline hazard function, corresponding to a subject with $z=0$, and~$\beta_0\in\R^p$ is the vector of the regression coefficients.

Let $H$ and $H^{uc}$ denote respectively the distribution function of the follow-up time and the sub-distribution function of the uncensored observations, i.e.,
\begin{equation}
\label{eq:def Huc}
H^{uc}(x)=\p(T\leq x,\Delta=1)=\int \delta\1_{\{t\leq x\}}\,\mathrm{d}\mathbb{P}(t,\delta,z),
\end{equation}
where $\p$ is the distribution of $(T,\Delta,Z)$.
We also require the following assumptions, some of which are common in large sample studies of the Cox model:
\begin{itemize}
\item[(A1)]
Let $\tau_F,\,\tau_G$ and $\tau_H$ be the end points of the support of $F,\,G$ and $H$. Then
\[
\tau_H=\tau_G<\tau_F\leq\infty.
\]
\item[(A2)]
There exists $\epsilon>0$ such that
\[
\sup_{|\beta-\beta_0|\leq\epsilon}\E\left[|Z|^2
\mathrm{e}^{2\beta'Z}
\right]<\infty.
\]
\end{itemize}
Let us briefly comment on these assumptions.
While the first one tells us that, at the end of the study there is at least one subject alive,
the second one is somewhat hard to justify from a practical point of view.
Condition (A2) was used in~\cite{Tsiatis81}, to establish asymptotic normality of $\hat\beta_n$,
and in~\cite{LopuhaaNane2013}, to ensure a squared integrable envelope for certain classes of functions
when using empirical process theory.
We require~(A2) essentially to apply results from~\cite{Tsiatis81} and~\cite{LopuhaaNane2013},
but this condition is also useful to bound averages that involve differences of the type
$\exp\{\hat\beta_n'Z_i\}-\exp\{\beta_0'Z_i\}$.
One can think of (A2) as a condition on the boundedness of the second moment of the covariates,
uniformly for $\beta$ in a neighborhood of~$\beta_0$.
Although, at first sight, it seems complicated, condition~(A2) is easy to verify in some important cases such as bounded covariates.

By now, it seems to be rather a standard choice to estimate $\beta_0$ in~\eqref{eq:cox model} by $\hat{\beta}_n$, the maximizer of the partial likelihood function
\[
L(\beta)
=
\prod_{i=1}^m
\frac{\mathrm{e}^{\beta'Z_i}}{\sum_{j=1}^n
\1_{\{T_j\geq X_{(i)}\}}\mathrm{e}^{\beta'Z_j}}
\]
as proposed in~\cite{Cox72} and~\cite{Cox75},
where $0<X_{(1)}<\cdots<X_{(m)}<\infty$ denote the ordered, observed event times.
The asymptotic behavior of $\hat{\beta}_n$ was first studied by~\cite{Tsiatis81}.
We aim at estimating $\lambda_0$, subject to the constraint that it is increasing (the case of a decreasing hazard is analogous), on the basis of $n$ observations $(T_1,\Delta_1,Z_1),\dots,(T_n,\Delta_n,Z_n)$.  We  refer to the quantity
\[
\Lambda_0(t)=\int_0^t\lambda_0(u)\,\mathrm{d}u,
\]
as the cumulative baseline hazard and, by introducing
\begin{equation}
\label{eq:def Phi}
\Phi(x;\beta)=\int \1_{ \{t\geq x\}}\,\mathrm{e}^{\beta'z}\,\mathrm{d}\p(t,\delta,z),
\end{equation}
we have
\begin{equation}
\label{eqn:lambda0}
\lambda_0(x)
=
\frac{h(x)}{\Phi(x;\beta_0)},
\end{equation}
where $h(x)=\mathrm{d}H^{uc}(x)/\mathrm{d}x$
(e.g., see (9) in~\cite{LopuhaaNane2013}).
For $\beta\in\R^p$ and $x\in\R$, the function $\Phi(x;\beta)$ can be estimated by
\begin{equation}
\label{eq:def Phin}
\Phi_n(x;\beta)=\int \1_{\{t\geq x\}} \mathrm{e}^{\beta'z}\,\mathrm{d}\p_n(t,\delta,z),
\end{equation}
where $\p_n$ is the empirical measure of the triplets $(T_i,\Delta_i,Z_i)$ with $i=1,\dots,n.$  Moreover, in Lemma 4 of~\cite{LopuhaaNane2013} it is shown that
\begin{equation}
\label{eqn:Phi}
\sup_{x\in\R}|\Phi_n(x;\beta_0)-\Phi(x;\beta_0)|=O_p(n^{-1/2}).
\end{equation}
It will often be used throughout the paper that a stochastic bound of the same order also holds
for the distance between
the cumulative baseline hazard $\Lambda_0$ and
the Breslow estimator
\begin{equation}
\label{eq:Breslow}
\Lambda_n(x)=\int \frac{\delta\1_{\{ t\leq x\}}}{\Phi_n(t;\hat{\beta}_n)}\,\mathrm{d}\p_n(t,\delta,z),
\end{equation}
but only on intervals staying away from the right boundary, i.e.,
\begin{equation}
\label{eqn:Breslow}
\sup_{x\in[0,M]}|\Lambda_n(x)-\Lambda_0(x)|=O_p(n^{-1/2}), \qquad\text{for all }0<M<\tau_H,
\end{equation}
(see Theorem 5 in~\cite{LopuhaaNane2013}).

Smoothing is done by means of kernel functions.
We will consider kernel functions $k$ that are $m$-orthogonal, for some $m\geq 1$,
which means that~$\int |k(u)||u|^m\,\mathrm{d}u<\infty$ and
$\int k(u)u^j\,\mathrm{d}u=0$, for $j=1,\ldots,m-1$, if $m\geq 2$.
We assume that
\begin{equation}
\label{def:kernel}
\begin{split}
&
\text{$k$ has bounded support $[-1,1]$ and is such that $\int_{-1}^1 k(y)\,\mathrm{d}y=1$;}\\
&
\text{$k$ is twice differentiable with a bounded derivative.}
\end{split}
\end{equation}
We denote by $k_b$ its scaled version $k_b(u)=b^{-1}k(u/b)$.
Here $b=b_n$ is a bandwidth that depends on the sample size, in such a way that
$0<b_n\to 0$ and $nb_n\to\infty$, as $n\to\infty$.
From now on, we will simply write $b$ instead of $b_n$.
Note that if $m>2$, the kernel function $k$ necessarily attains negative values and as a result also the smooth estimators of the baseline hazard defined in
Sections~\ref{sec:MSLE} and~\ref{sec:GS} may be negative.
To avoid this, one could restrict oneself to $m=2$.
In that case, the most common choice is to let $k$ be a symmetric probability density.


\section{Maximum smooth likelihood estimator}
\label{sec:MSLE}
Maximum smoothed likelihood estimation is studied in~\cite{eggermont-lariccia2000},
who obtain $L_1$-error bounds for the maximum smoothed likelihood estimator of a monotone density.
This method was also considered in~\cite{GJW10} for estimating the distribution function of interval censored observations.
The approach is to smooth the loglikelihood and then maximize the smoothed loglikelihood over all monotone functions of interest.
For a fixed $\beta$, the (pseudo) loglikelihood for the Cox model can be expressed as
\begin{equation}
\label{def:loglikelihood cox}
\int
\left(
\delta\log \lambda_0(t)-\mathrm{e}^{\beta'z}\int_0^t\lambda_0(u)\,\mathrm{d}u
\right)\,\mathrm{d}\p_n(t,\delta,z),
\end{equation}
(see (2) in~\cite{LopuhaaNane2013}).
To construct the maximum smoothed likelihood estimator (MSLE) we replace~$\p_n$ in the previous expression with the smoothed empirical measure (in the time direction),
\[
\mathrm{d}\tilde{\p}_n(t,\delta,z)
=
\frac{1}{n}\sum_{i=1}^n \1_{(\Delta_i,Z_i)}(\delta,z)\,k_b(t-T_i)\,\mathrm{d}t,
\]
and then maximize the smoothed (pseudo) loglikelihood
\begin{equation}
\label{def:smooth likelihood}
\ell^s_\beta(\lambda_0)
=
\int
\left(
\delta\log \lambda_0(t)-\mathrm{e}^{\beta'z}\int_0^t\lambda_0(u)\,\mathrm{d}u
\right)\,\mathrm{d}\tilde{\p}_n(t,\delta,z).
\end{equation}
The characterization of the MSLE is similar to that of the ordinary MLE
(see Lemma~1 in~\cite{LopuhaaNane2013}).
It involves the following processes.
Fix $\beta\in\R^p$ and let
\begin{equation}
\label{eqn:v_n w_n}
\begin{split}
w_n(t;\beta)
&=
\frac{1}{n}\sum_{i=1}^n
\mathrm{e}^{\beta'Z_i}\int_{t}^\infty k_b(u-T_i)\,\mathrm{d}u,\\
v_n(t)
&=
\frac{1}{n}
\sum_{i=1}^n \Delta_ik_b(t-T_i).
\end{split}
\end{equation}
The next lemma characterizes the maximizer of $\ell^s_\beta$.
The proof can be found in Appendix~\ref{subsec:proofs MSLE}.

\begin{lemma}
\label{lem:char MSLE}
Let $\ell^s_\beta$, $w_n$ and $v_n$ be defined by~\eqref{def:smooth likelihood} and~\eqref{eqn:v_n w_n}, respectively.
The unique maximizer of $\ell^s_\beta$ over all nondecreasing positive functions $\lambda_0$ can be described as the slope of the greatest convex minorant (GCM)
of the continuous cumulative sum diagram
\begin{equation}
\label{eqn:graph}
t\mapsto \left(\int_0^t w_n(x;\beta)\,\mathrm{d}x, \int_0^t v_n(x)\,\mathrm{d}x\right),
\qquad t\in[0,\tau_\beta],
\end{equation}
where $\tau_\beta
=
\sup\{t\geq 0:w_n(t;\beta)>0\}$.
\end{lemma}

For a fixed $\beta$, let $\hat{\lambda}^s_n(x;\beta)$ be the unique maximizer of $\ell^s_\beta(\lambda_0)$ over all nondecreasing positive functions $\lambda_0$.
We define the MSLE by
\begin{equation}
\label{def:lambdaMS}
\hat{\lambda}^{MS}_n(x)=\hat{\lambda}^s_n(x;\hat{\beta}_n),
\end{equation}
where $\hat{\beta}_n$ denotes the maximum partial likelihood estimator for $\beta_0$.
It can be seen that under appropriate smoothness assumptions,
\[
\begin{split}
\int_0^t w_n(x;\hat\beta_n)\,\mathrm{d}x
&=
\int
\hat W_n(s)k_b(t-s)\,\mathrm{d}s+O_p(n^{-1/2})+O_p(b),\\
\int_0^t v_n(x)\,\mathrm{d}x
&=
\int V_n(s)k_b(t-s)\,\mathrm{d}s+O_p(b),
\end{split}
\]
where the processes $V_n$ and $\hat W_n$, as defined in~Lemma~1 in~\cite{LopuhaaNane2013}, determine the cumulative sum diagram corresponding
to the ordinary MLE.
This means that the cumulative sumdiagram that characterizes the MSLE,
is asymptotically equivalent to a kernel smoothed version of the cumulative sumdiagram that characterizes the ordinary MLE.

As can be seen from the proof of Lemma~\ref{lem:char MSLE},
the MSLE minimizes
\begin{equation}
\label{eqn:isot.regr}
\psi(\lambda)=\frac{1}{2}\int \left(\lambda(x)-\frac{v_n(x)}{w_n(x;\hat\beta_n)}\right)^2
w_n(x;\beta)\,\mathrm{d}x,
\end{equation}
over all nondecreasing functions $\lambda$.
This suggests
\begin{equation}
\label{def:naive est MSLE}
\hat{\lambda}_n^{\mathrm{naive}}(x)=\frac{v_n(x)}{w_n(x;\hat{\beta}_n)}
\end{equation}
as a naive estimator for $\lambda_0$.
The naive estimator is the ratio of two smooth functions, being the
derivatives of the vertical and horizontal processes in the continuous cumulative sum diagram in~\eqref{eqn:graph}.
The naive estimator is smooth, but not necessarily monotone and its weighted least squares projection is the MSLE.
Figure~\ref{fig:MSLE}
\begin{figure}[t]
\includegraphics[width=\textwidth]{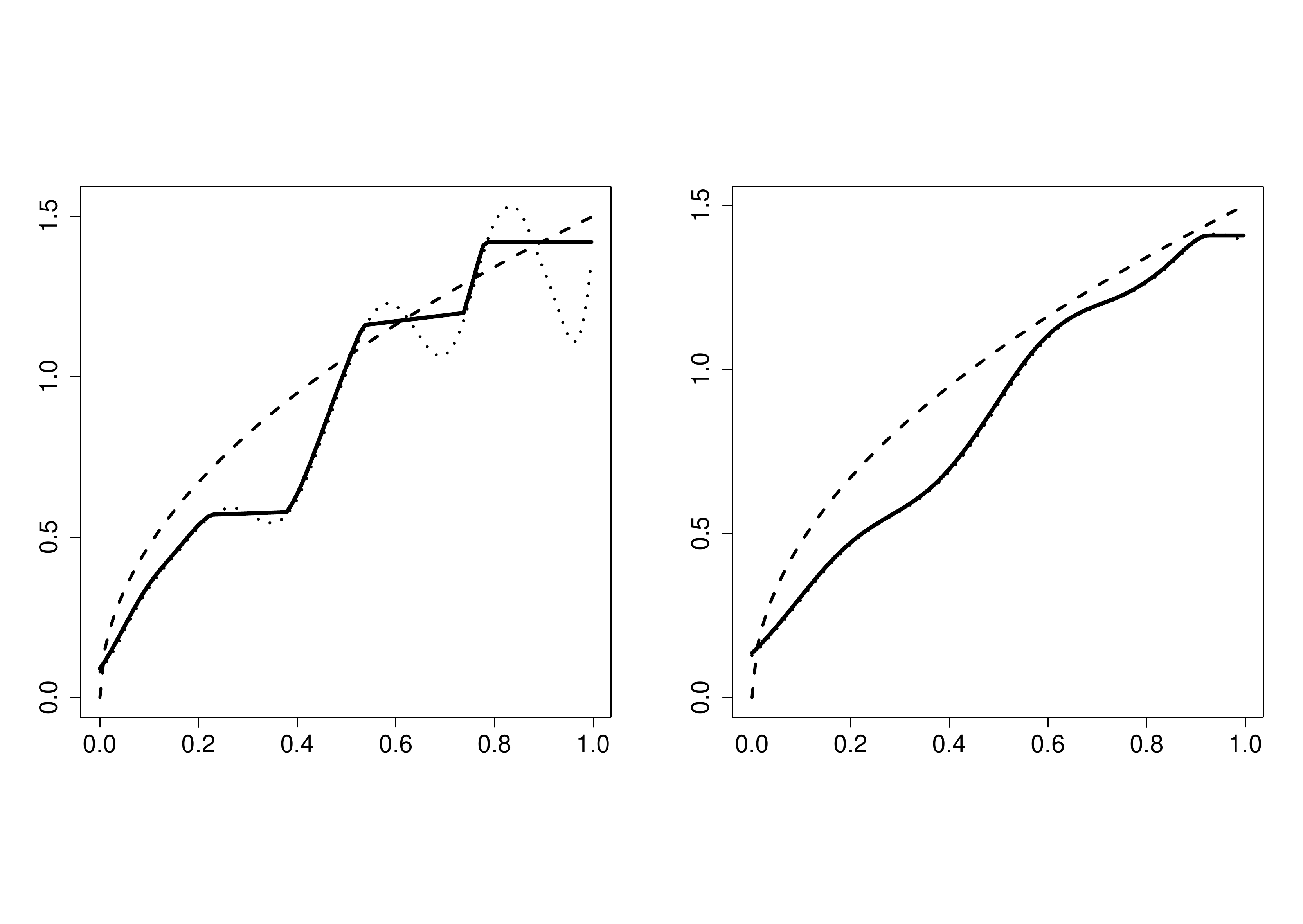}
\caption{Left panel: The MSLE (solid) and the naive estimator (dotted) of the hazard function (dashed) using bandwidth $b_n=0.5n^{-1/5}$.
Right panel: The same but using bandwidth $b_n=n^{-1/5}$.}
\label{fig:MSLE}
\end{figure}
illustrates the MSLE and the naive estimator for a sample of size $n=500$ from a Weibull baseline distribution with shape parameter $1.5$ and scale $1$.
For simplicity, the covariate and the censoring time are chosen to be uniformly $(0,1)$ distributed and we take $\beta_0=0.5$.
We used the triweight kernel function
$k(u)=(35/32)(1-u^2)^3\1_{\{|u|\leq 1\}}$
and bandwidth $b=n^{-1/5}$.
Note that if we use bandwidth $b_n=0.5n^{-1/5}$, the naive estimator is not monotone,
but the distance to the MSLE (which is the isotonic version of $\hat{\lambda}^{\mathrm{naive}}_n$) is very small.
On the other hand, for bandwidth $b_n=n^{-1/5}$ isotonization is not needed and the two estimators coincide.
Indeed, following the reasoning in~\cite{GJW10}, the derivation of the asymptotic distribution of $\hat{\lambda}^{MS}_n$
is based on showing that with probability converging to one,
the naive estimator will be monotone and equal to $\hat{\lambda}^{MS}_n$ on large intervals.
Consequently, it will be sufficient to find the asymptotic distribution of the naive estimator.
The advantage of this approach is that in this way we basically have to deal with the naive estimator, which is a more tractable process.

This approach applies more generally.
The situation for the MSLE is a special case of the more general situation,
where the isotonic estimator is the derivative $\mathrm{d} \widehat{Y}_n/\mathrm{d} X_n$
of the greatest convex minorant
$\{(X_n(t),\widehat{Y}_n(t)):t\in[0,\hat\tau]\}$ of the graph
$\{(X_n(t),Y_n(t)):t\in[0,\hat\tau]\}$, for some~$0<\hat\tau<\tau_H$,
where $X_n$ and $Y_n$ are differentiable processes in a
cumulative sumdiagram, whereas the naive estimator is the ratio~$\mathrm{d} Y_n/\mathrm{d} X_n$
of the derivatives of $X_n$ and $Y_n$.
The MSLE and the corresponding naive estimator from~\eqref{def:naive est MSLE} form a special case, with
$X_n=\tilde W_n$, $Y_n=\tilde V_n$, where
\begin{equation}
\label{def:tilde Wn tilde V_n}
\tilde W_n(t)=\int_0^t w_n(x;\hat{\beta}_n)\,\mathrm{d}x,
\qquad
\tilde V_n(t)=\int_0^t v_n(x)\,\mathrm{d}x,
\end{equation}
and $\hat\tau=\sup\{t\geq 0:w_n(t;\hat\beta_n)>0\}$.
The following result considers the general setup and shows that, in that case, the isotonic estimator and the corresponding naive estimator
coincide on large intervals with probability tending to one.
\begin{lemma}
\label{lem:coincide}
Let $X_n$ and $Y_n$ be differentiable processes and let $\{(X_n(t),\widehat{Y}_n(t)):t\in[0,\hat\tau]\}$ be the greatest convex minorant
of the graph $\{(X_n(t),Y_n(t)):t\in[0,\hat\tau]\}$, for some $0<\hat\tau<\tau_H$.
Let $\lambda_n^{IS}(t)=\mathrm{d}\widehat{Y}_n(t)/\mathrm{d}X_n(t)$ and
$\hat\lambda_n(t)=\mathrm{d}Y_n(t)/\mathrm{d}X_n(t)$,
for $t\in[0,\hat\tau]$.
Suppose that
\begin{enumerate}[(a)]
\item
\label{cond a}
$X_n(s)\leq X_n(t)$, for $0\leq s\leq t\leq \hat\tau$;
\item
\label{cond b}
for every $t\in(0,\hat\tau)$ fixed, $\hat\lambda_n(t)\to \lambda_0(t)$, in probability;
\item
\label{cond c}
for all $0<\ell<M<\hat\tau$ fixed,
$\mathbb{\mathbb{P}}\big(\hat{\lambda}_n\text{ is increasing on } [\ell,M]\big)\to 1$;
\item
\label{cond d}
there exists processes $X_0$ and $Y_0$, such that
\[
\sup_{t\in[0,\hat\tau]}
\left|X_n(t)-X_0(t)\right|\xrightarrow{P}0,
\qquad
\sup_{t\in[0,\hat\tau]}
\left|Y_n(t)-Y_0(t)\right|\xrightarrow{P}0.
\]
Moreover, the process $X_0$ is absolutely continuous with a strictly positive nonincreasing derivative $x_0$, and $X_0$ and $Y_0$ are related by
$Y_0(t)=\int_0^t \lambda_0(u)\,\mathrm{d}X_0(u)$.
\end{enumerate}
Then, for all $0<\ell<M<\hat\tau$, $\p
\left(
\hat{\lambda}_n(t)=\hat{\lambda}_n^{MS}(t),\text{ for all } t\in[\ell,M]
\right)\to 1$.
\end{lemma}
The proof of Lemma~\ref{lem:coincide} can be found in the Appendix~\ref{subsec:proofs MSLE}.
We will apply Lemma~\ref{lem:coincide} to the~MSLE and the naive estimator from~\eqref{def:naive est MSLE}.
Recall that~$\hat{\lambda}^{MS}_n$ and~$\hat{\lambda}_n^{\mathrm{naive}}$ are defined on $[0,\hat\tau_n]$, where
$\hat\tau_n=\sup\{t\geq 0:w_n(t;\hat{\beta}_n)>0\}$, and note that~$\hat\tau_n\to \tau_H$ with probability one.
Condition~\eqref{cond a} of Lemma~\ref{lem:coincide} is trivially fulfilled with $X_n=\tilde W_n$ defined in~\eqref{def:tilde Wn tilde V_n}.
A first key result is that for each $0<\ell<M<\tau_H$, it holds
\begin{equation}
\label{eqn:approx vn wn}
\begin{split}
\sup_{t\in[\ell,M]}|v_n(t)-h(t)|
&=
O(b^m)+O_p(b^{-1}n^{-1/2}),\\
\sup_{t\in[\ell,M]}|w_n(t;\hat{\beta}_n)-\Phi(t;\beta_0)|
&=
O(b^m)+O_p(b^{-1}n^{-1/2}),
\end{split}
\end{equation}
where $v_n$, $w_n$ and $\Phi$ are defined in~\eqref{eqn:v_n w_n} and~\eqref{eq:def Phi},
see Lemma~\ref{le:1}.
A direct consequence of~\eqref{eqn:approx vn wn} is the fact that the naive estimator converges to~$\lambda_0$ uniformly on
compact intervals within the support, as long as $b\to0$ and $1/b=o(n^{1/2})$, see Lemma~\ref{lem:cons-naive}.
In particular, this will ensure condition~\eqref{cond b} of Lemma~\ref{lem:coincide}.
A second key result is that, under suitable smoothness conditions,
for each $0<\ell<M<\tau_H$, it holds
\begin{equation}
\label{eqn:approx vn' wn'}
\begin{split}
\sup_{t\in[\ell,M]}|v'_n(t)-h'(t)|
&\xrightarrow{\p}0,\\
\sup_{t\in[\ell,M]}|w'_n(t;\hat{\beta}_n)-\Phi'(t;\beta_0)|
&
\xrightarrow{\p}0,
\end{split}
\end{equation}
where $v_n$, $w_n$ and $\Phi$ are defined in~\eqref{eqn:v_n w_n} and~\eqref{eq:def Phi},
see Lemma~\ref{le:1a}.
This will imply that the naive estimator is increasing on large intervals
with probability tending to one, see Lemma~\ref{lem:monotone},
which yields condition~\eqref{cond c} of Lemma~\ref{lem:coincide}.
Finally, condition~\eqref{cond d} of Lemma~\ref{lem:coincide} is shown to hold
with $X_0=H^{uc}$ from~\eqref{eq:def Huc} and $Y_0=W_0$, defined by
\begin{equation}
\label{def:W0}
W_0(t)=\int_0^t \Phi(x;\beta_0)\,\mathrm{d}x.
\end{equation}
In view of~~\eqref{eqn:approx vn wn} and~\eqref{eqn:lambda0},
this is to be expected, and it is made precise in Lemma~\ref{le:2}.
Hence, Lemma~\ref{lem:coincide} applies to the MSLE and the naive estimator from~\eqref{def:naive est MSLE}.
Therefore we have the following corollary.
\begin{cor}
\label{cor:MSLE=naive}
Suppose that (A1)-(A2) hold.
Let $H^{uc}(t)$ and $\Phi(t;\beta_0)$ be defined in~\eqref{eq:def Huc} and~\eqref{eq:def Phi}, and let $h(t)=\mathrm{d}H^{uc}(t)/\mathrm{d}t$,
satisfying~\eqref{eqn:lambda0}.
Suppose that $h$ and $t\mapsto\Phi(t;\beta_0)$ are continuously differentiable,
and that~$\lambda'_0$ is uniformly bounded from below by a strictly positive constant.
Let $k$ satisfy~\eqref{def:kernel} and let~$\hat{\lambda}_n^{\mathrm{naive}}$ be defined in~\eqref{def:naive est MSLE}.
If $b\to0$ and $1/b=O(n^{\alpha})$, for some $\alpha\in(0,1/4)$, then for each $0<\ell<M<\tau_H$,
\[
\p
\left(
\hat{\lambda}_n^{\mathrm{naive}}(x)=\hat{\lambda}_n^{MS}(x),\text{ for all } x\in[\ell,M]
\right)\to 1.
\]
Consequently, for all $x\in(0,\tau_H),$ the asymptotic distributions of $\hat{\lambda}_n^{\mathrm{naive}}(x)$ and $\hat{\lambda}_n^{MS}(x)$ are the same.
\end{cor}

Under similar smoothness conditions as needed to obtain~\eqref{eqn:approx vn wn}, see Lemma~\ref{le:1}, one can show that
\begin{equation}
\label{eqn:lemma1-2a}
\begin{split}
\sup_{t\in[\ell,M]}|v'_n(t)-h'(t)|
&=
O(b^{m-1})+O_p(b^{-2}n^{-1/2}),\\
\sup_{t\in[\ell,M]}|w'_n(t;\hat{\beta}_n)-\Phi'(t;\beta_0)|
&=
O(b^{m-1})+O_p(b^{-1}n^{-1/2}).
\end{split}
\end{equation}
In that case, it can also be proved that
\[
\sup_{x\in[\ell,M]}\left|\frac{\mathrm{d}}{\mathrm{d}x}\hat{\lambda}_n^{\mathrm{naive}}(x)-\lambda'_0(x)\right|
=
O(b^{m-1})+O_p(b^{-2}n^{-1/2})=o_P(1),
\]
as long as $b\to0$ and $1/b^2=o(n^{1/2})$.
One would expect that if instead of a standard kernel we use a boundary corrected version,
then~\eqref{eqn:approx vn' wn'} would hold on the whole support~$[0,\tau_H]$
and consequently we would obtain that the naive estimator is monotone on $[0,\tau_H]$ with probability tending to one.
However, the use of boundary kernels makes the computations much more complicated.
Nevertheless, monotonicity on intervals $[\ell,M]$ is enough for our purposes,
because we aim at finding the pointwise asymptotic distribution at the interior of the support.

From Corollary~\ref{cor:MSLE=naive}, together with the fact that the naive estimator converges to~$\lambda_0$ uniformly on compact intervals within the support,
see Lemma~\ref{lem:cons-naive}, another consequence of Lemma~\ref{lem:coincide} is the following corollary
concerning uniform convergence of the MSLE.
\begin{cor}
\label{cor:cons-MS}
Suppose that (A1)-(A2) hold.
Let $H^{uc}(t)$ and $\Phi(t;\beta_0)$ be defined in~\eqref{eq:def Huc} and~\eqref{eq:def Phi}, and let $h(t)=\mathrm{d}H^{uc}(t)/\mathrm{d}t$,
satisfying~\eqref{eqn:lambda0}.
Suppose that $h$ and $t\mapsto\Phi(t;\beta_0)$ are $m\geq 1$ times continuously differentiable,
and that $\lambda_0'$ is uniformly bounded from below by a strictly positive constant.
Let $k$ be $m$-orthogonal satisfying~\eqref{def:kernel}.
Then, the maximum smooth likelihood estimator is uniformly consistent on compact intervals $[\ell,M]\subset (0,\tau_H)$:
\[
\sup_{x\in[\ell,M]}\left|\hat{\lambda}^{MS}_n(x)-\lambda_0(x)\right|
=
O(b^{m})+O_p(b^{-1}n^{-1/2}).
\]
\end{cor}
\begin{proof}
The result follows immediately from Corollary~\ref{cor:MSLE=naive} and Lemma~\ref{lem:cons-naive}.
\end{proof}
To obtain the asymptotic distribution of $\hat{\lambda}^{MS}_n(x)$, we first obtain the asymptotic distribution of $\hat{\lambda}_n^{\mathrm{naive}}(x)$.
To this end we establish the joined asymptotic distribution of the vector~$(w_n(x;\hat\beta_n),v_n(x))$,
see Lemma~\ref{lemma:distr}.
Then an application of the delta-method yields the limit distribution of~$\hat{\lambda}_n^{\mathrm{naive}}$
as well as that of~$\hat{\lambda}^{MS}_n$, due to Corollary~\ref{cor:MSLE=naive}.

\begin{theo}
\label{theo:distrMS}
Suppose that (A1)-(A2) hold and fix $x\in(0,\tau_H)$.
Let $H^{uc}(t)$ and $\Phi(t;\beta_0)$ be defined in~\eqref{eq:def Huc} and~\eqref{eq:def Phi},
and let $h(t)=\mathrm{d}H^{uc}(t)/\mathrm{d}t$, satisfying~\eqref{eqn:lambda0}.
Suppose that $h$ and $t\mapsto\Phi(t;\beta_0)$ are $m\geq 2$ times continuously differentiable
and let $k$ be $m$-orthogonal satisfying~\eqref{def:kernel}.
Let~$\hat{\lambda}^{MS}_n(x)$ be defined in~\eqref{def:lambdaMS} and assume that $n^{1/(2m+1)}b\to c>0$.
Then, for each $x\in(0,\tau_H)$, the following holds
\[
n^{m/(2m+1)}
\left(\hat{\lambda}^{MS}_n(x)-\lambda_0(x)\right)\xrightarrow{d} N(\mu,\sigma^2),
\]
where
\begin{equation}
\label{def:mu sigma}
\begin{split}
\mu
&=
\frac{(-c)^m}{m!}
\frac{h^{(m)}(x)-\lambda_0(x)\Phi^{(m)}(x;\beta_0)}{\Phi(x;\beta_0)}
\int_{-1}^1 k(y)y^m\,\mathrm{d}y;\\
\sigma^2
&=
\frac{\lambda_0(x)}{c\Phi(x;\beta_0)}\int_{-1}^1 k^2(y)\,\mathrm{d}y.
\end{split}
\end{equation}
This also holds if we replace $\hat{\lambda}^{MS}_n(x)$ with $\hat{\lambda}_n^{\mathrm{naive}}(x)$, as defined in~\eqref{def:naive est MSLE}.
\end{theo}
The proof of Theorem~\ref{theo:distrMS} can be found in Appendix~\ref{subsec:proofs MSLE}.
Theorem~\ref{theo:distrMS} is comparable to Theorem~3.5 in~\cite{LopuhaaMustaSI2016},
where the limiting normal distribution of the smoothed maximum likelihood estimator $\hat{\lambda}^{SM}_n(x)$
and the smoothed Grenander-type estimator $\tilde\lambda_n^{SG}(x)$ is established, i.e.,
\begin{equation}
\label{eq:asymp norm SM SG}
\begin{split}
n^{m/(2m+1)}
\left(\hat{\lambda}^{SM}_n(x)-\lambda_0(x)\right)
&\xrightarrow{d}
N(\widetilde{\mu},\sigma^2),\\
n^{m/(2m+1)}
\left(
\tilde\lambda_n^{SG}(x)-\hat{\lambda}^{SM}_n(x)
\right)
&\xrightarrow{\p} 0,
\end{split}
\end{equation}
where
\begin{equation}
\label{def:mutilde}
\widetilde{\mu}
=\frac{(-c)^m}{m!}
\lambda_0^{(m)}(x)
\int_{-1}^1 k(y)y^m\,\mathrm{d}y.
\end{equation}%
The limiting variance is the same, but the asymptotic mean is shifted.
A natural question is whether $\hat{\lambda}^{MS}_n(x)$ is asymptotically equivalent to these estimators,
if we correct for the difference in the asymptotic mean.
The next theorem shows that this is indeed the case.
The proof can be found in Appendix~\ref{subsec:proofs MSLE}.
In order to use results from~\cite{LopuhaaMustaSI2016}, we have to strengthen condition (A2) slightly.
\begin{theo}
\label{theo:MSLE asymptotic equivalence}
Suppose that (A1) holds and
\begin{itemize}
\item[(A2')]
there exists $\epsilon>0$ such that
\[
\sup_{|\beta-\beta_0|\leq\epsilon}\E
\left[
|Z|^2
\left(
\mathrm{e}^{2\beta'Z}+\mathrm{e}^{4\beta'Z}
\right)
\right]<\infty.
\]
\end{itemize}%
Fix $x\in(0,\tau_H)$.
Suppose that $\lambda_0$ and $t\mapsto\Phi(t;\beta_0)$ are $m\geq 2$ times continuously differentiable,
with $\lambda'_0$ uniformly bounded from below by a strictly positive constant,
and let $k$ be $m$-orthogonal satisfying~\eqref{def:kernel}.
Let~$\hat{\lambda}^{MS}_n(x)$ be the maximum smoothed likelihood estimator and let
$\hat{\lambda}^{SM}(x)$ be the smoothed maximum likelihood estimator, defined in~\cite{LopuhaaMustaSI2016}.
Let~$\mu$ and~$\widetilde{\mu}$ be defined in~\eqref{def:mu sigma} and~\eqref{def:mutilde}, respectively.
Then, for each $x\in(0,\tau_H)$, the following holds
\[
n^{m/(2m+1)}
\left(
\hat\lambda_n^{MS}(x)
-
\hat\lambda_n^{SM}(x)
\right)
-
\left(
\mu-\widetilde{\mu}
\right)
\to
0
\]
in probability, and similarly if we replace~$\hat\lambda_n^{SM}(x)$
by the smoothed Grenander-type estimator~$\tilde\lambda_n^{SG}(x)$, defined in~\cite{LopuhaaMustaSI2016}.
\end{theo}

\section{Isotonized smoothed Breslow estimator}
\label{sec:GS}

The second method that we consider is an isotonized version of the smoothed Breslow estimator, defined by
\begin{equation}
\label{def:smoothed Breslow}
\Lambda_n^s(x)
=
\int k_b(x-u)\Lambda_n(u)\,\mathrm{d}u.
\end{equation}
In order to avoid problems at the right end of the support, we fix $0<\tau^*<\tau_H$ and consider estimation only on $[0,\tau^*]$.
A similar approach was considered in~\cite{GJ13}, when estimating a monotone hazard of uncensored observations.
The main reason in our setup is that in order to exploit the representation in~\eqref{eqn:lambda0},
we must have $x<\tau_H$, because $\Phi(x;\beta_0)=0$ otherwise.
The isotonized smoothed Breslow estimator (ISBE) of a nondecreasing baseline hazard is a Grenander-type estimator,
as being defined as the left derivative of the greatest convex minorant of $\Lambda_n^s$ on $[0,\tau^*]$.
We denote this estimator by $\tilde{\lambda}_n^{GS}$.

Note that this type of estimator was defined also in~\cite{Nane} without the restriction on $[0,\tau^*]$.
Strong pointwise consistency was proved  and uniform consistency on intervals $[\epsilon, \tau_H-\epsilon]\subset[0,\tau_H]$ follows immediately from the monotonicity and the continuity of $\lambda_0$.
These results also illustrate that there are consistency problems at the end point of the support.
Since in practice we do not even know $\tau_H$, the choice of $\tau^*$ might be an issue.
Since one wants $\tau^*$ to be close to $\tau_H$, one reasonable choice would be to take as $\tau^*$ the $95\%$-empirical quantile of the follow-up times, because this converges to the theoretical $95\%$-quantile,
which is strictly smaller than $\tau_H$.
Note that we cannot choose~$T_{(n)}$, because it converges to $\tau_H$,
i.e., for large $n$, it will be greater than any fixed $\tau^*<\tau_H$.

Figure~\ref{fig:GS} shows the smoothed Breslow estimator and the ISBE for the same sample as in Figure~\ref{fig:MSLE}.
To avoid problems at the boundary we use the boundary corrected version of the kernel function and consider the data up to the $95\%$-empirical quantile of the follow-up times.
The bandwidth is $b_n=n^{-1/5}$.
\begin{figure}[t]
\includegraphics[width=\textwidth]{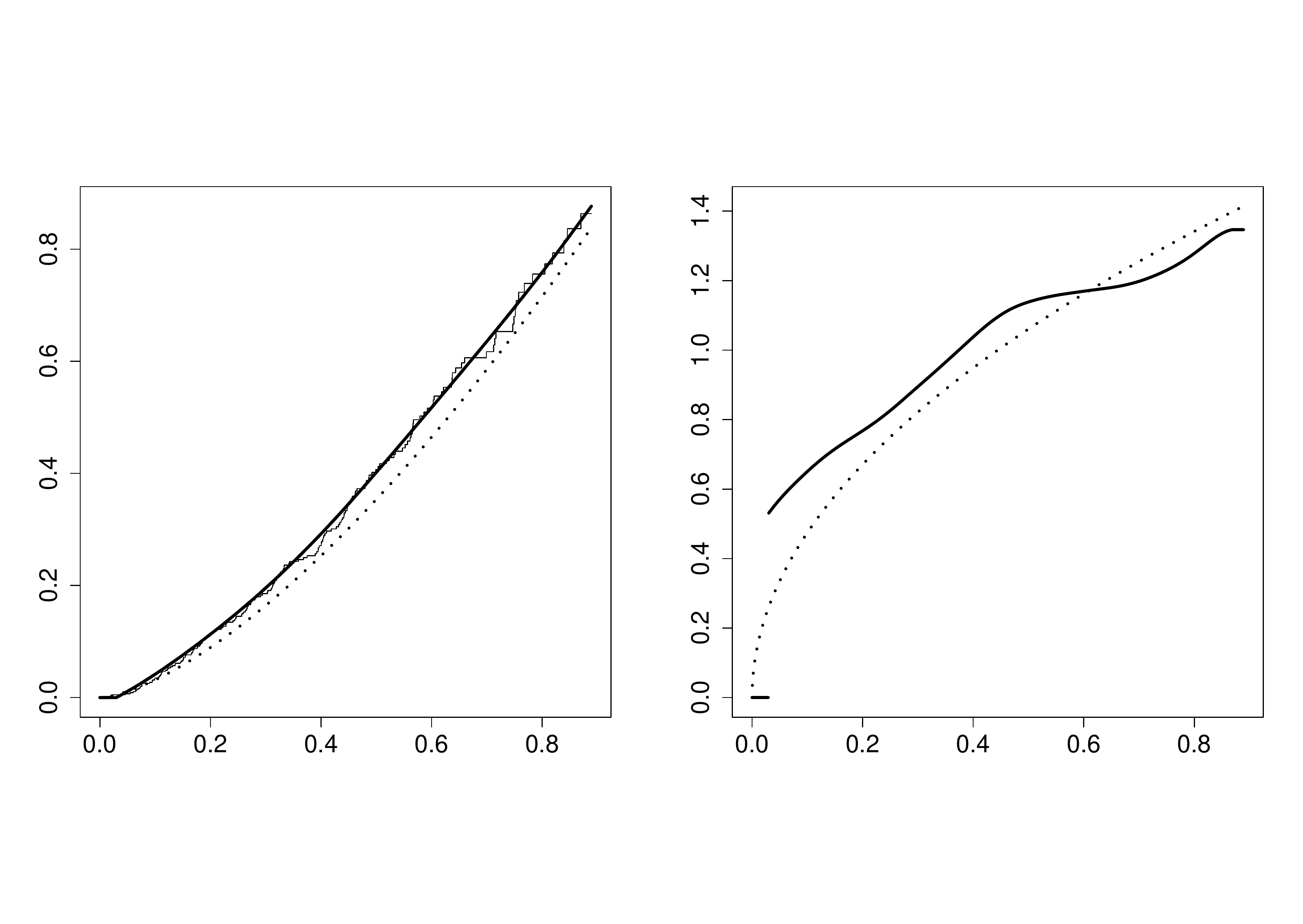}
\caption{Left panel:
The smoothed version (solid) of the Breslow estimator (solid-step function) for the cumulative baseline hazard (dotted)
and the greatest convex minorant (dashed).
Right panel: The Grenander-type smoothed estimator (solid) of the baseline hazard (dotted).}
\label{fig:GS}
\end{figure}
Similar to the proof of Lemma~\ref{lem:char MSLE}, it follows from Lemma 1 in~\cite{GJ10},
that $\tilde{\lambda}_n^{GS}$ is continuous and is the unique maximizer of
\[
\psi(\lambda)=\frac{1}{2}
\int_0^{\tau^*}
\left(
\lambda(x)-\lambda_n^s(x)
\right)^2\,\mathrm{d}x
\]
over all nondecreasing functions $\lambda$, where
\begin{equation}
\label{def:vn GS}
\lambda_n^s(x)=\frac{\mathrm{d}}{\mathrm{d}x}\Lambda_n^s(x)=\int k'_b(x-u)\Lambda_n(u)\,\mathrm{d}u.
\end{equation}
This suggests
\begin{equation}
\label{def:naive est GS}
\tilde{\lambda}_n^{\mathrm{naive}}(x)=\lambda_n^s(x)
\end{equation}
as another naive estimator for $\lambda_0(x)$.
This naive estimator is the derivative of the smoothed Breslow.
Again, it is smooth but not necessarily monotone and its least squares projection is the ISBE.
Note that by means of integration by parts, we can also write
\[
\lambda_n^s(x)
=
\int k_b(x-u)\,\mathrm{d}\Lambda_n(u).
\]
Hence, the naive estimator from~\eqref{def:vn GS} is equal to the ordinary Rosenblatt-Parzen kernel estimator for the baseline hazard.
Asymptotic normality for this estimator under random censoring has been proven
by~\cite{ramlau-hansen1983} and~\cite{tanner-wong1983}.
A similar result in a general counting processes setup, that includes the Cox model, is stated in~\cite{wells1994},
but only the idea of the proof is provided.
We will establish asymptotic normality for the naive estimator from~\eqref{def:vn GS} in our current setup of the Cox model,
see the proof of Theorem~\ref{theo:as.distrGS}.

Then, similar to the approach used in Section~\ref{sec:MSLE}, the derivation of the asymptotic distribution
of~$\tilde{\lambda}^{GS}_n$ is based on showing that it is equal to the naive estimator in~\eqref{def:naive est GS}
on large intervals with probability converging to one.
The ISBE is a special case of Lemma~\ref{lem:coincide}, with $X_n(t)=t$, $Y_n(t)=\Lambda_n^s(t)$,
and $\hat\tau=\tau^*$.
As before, condition~\eqref{cond a} of Lemma~\ref{lem:coincide} is trivial
and condition~\eqref{cond b} is fairly straightforward, see~\eqref{eq:Lemma53 GS} for details.
Condition~\eqref{cond c} of Lemma~\ref{lem:coincide} is established in Lemma~\ref{lem:monotone2}
and condition~\eqref{cond d} is also straightforward, see~\eqref{eq:Lemma54 GS} for details.
Hence, Lemma~\ref{lem:coincide} applies to the ISBE and the naive estimator from~\eqref{def:naive est GS},
which leads to the following corollary.
\begin{cor}
\label{cor:ISBE=naive}
Suppose that (A1)-(A2) hold.
Let $\lambda_0$ be continuously differentiable, with $\lambda'_0$ uniformly bounded from below by a strictly positive constant,
and let $k$ satisfy~\eqref{def:kernel}.
If $b\to0$ and $1/b=O(n^{\alpha})$, for some $\alpha\in(0,1/4)$, then for each $0<\ell<M<\tau^*$,  it holds
\[
\p
\left(
\tilde{\lambda}_n^{\mathrm{naive}}(x)=\tilde{\lambda}_n^{GS}(x)\text{ for all } x\in[\ell,M]
\right)\to 1.
\]
Consequently, for all $x\in(0,\tau^*)$, the asymptotic distributions of $\tilde{\lambda}_n^{\mathrm{naive}}(x)$
and $\tilde{\lambda}_n^{GS}(x)$ are the same.
\end{cor}
The proof of Corollary~\ref{cor:ISBE=naive} can be found in Appendix~\ref{subsec:proofs GS}.

\begin{re}
Note that in case the kernel function is strictly positive on $(-1,1)$ and the baseline hazard is strictly increasing, one can easily check  that
\[
x\mapsto\int_{-1}^{x/b} k(y) \lambda_0(x-by)\,\mathrm{d}y
\]
is a continuously differentiable, strictly increasing function on $[0,M]$ and as a result we obtain that
\begin{equation}
\begin{split}
\frac{\mathrm{d}}{\mathrm{d}x}\tilde{\lambda}_n^{\mathrm{naive}}(x)
&=
\frac{\mathrm{d}}{\mathrm{d}x}\left(\int_{-1}^{x/b} k(y) \lambda_0(x-by)\,\mathrm{d}y  \right)
+
\frac{1}{b^2}\int k'\left(\frac{x-u}{b}\right)\,\mathrm{d}\left(\Lambda_n-\Lambda_0\right)(u)\\
&\geq C+o_P(1).
\end{split}
\end{equation}
This implies that $\tilde{\lambda}_n^{\mathrm{naive}}$ is increasing on $[0,M]$.
\end{re}
Finally, consistency and the asymptotic distribution of $\tilde{\lambda}_n^{GS}(x)$ is provided by the next theorem.
Its proof can be found in Appendix~\ref{subsec:proofs GS}.
\begin{theo}
\label{theo:as.distrGS}
Suppose that (A1)-(A2) hold and fix $x\in(0,\tau_H)$ and $\tau^*\in(x,\tau_H)$.
Assume that $\lambda_0$ is $m\geq2$ times continuously differentiable, with $\lambda'_0$ uniformly bounded from below by a strictly positive constant.
Assume that $t\mapsto \Phi(t;\beta_0)$ is continuous in a neighborhood of $x$
and let~$k$ be $m$-orthogonal satisfying~\eqref{def:kernel}.
Let~$\tilde{\lambda}_n^{GS}$ be the left derivative of the greatest convex minorant on $[0,\tau^*]$
of $\Lambda_n^s$ defined in~\eqref{def:smoothed Breslow}
and suppose that $n^{1/(2m+1)}b\to c>0$.
Then, for all $0<\ell<M<\tau^*$,
\[
\sup_{x\in[\ell,M]}
\left|
\tilde{\lambda}_n^{GS}(x)-\lambda_0(x)
\right|
=
O(b^m)+O_p(b^{-1}n^{-1/2}),
\]
in probability, and it holds that
\[
n^{m/(2m+1)}
\left(
\tilde{\lambda}_n^{GS}(x)-\lambda_0(x)
\right)
\xrightarrow{d}N(\mu,\sigma^2),
\]
where
\[
\mu
=
\frac{(-c)^m}{m!}
\lambda_0^{(m)}(x)
\int_{-1}^1 k(y)y^m\,\mathrm{d}y
\quad\text{ and }\quad
\sigma^2
=
\frac{\lambda_0(x)}{c\Phi(x;\beta_0)}
\int_{-1}^1 k(y)^2\,\mathrm{d}y.
\]
\end{theo}
According to Corollary~\ref{cor:ISBE=naive}, the naive estimator from~\eqref{def:naive est GS} has
the same limiting distribution described in Theorem~\ref{theo:as.distrGS}.
In this case we recover a result similar to Theorem~3.2 in~\cite{wells1994}.
As can be seen from~\eqref{eq:asymp norm SM SG} and~\eqref{def:mutilde}, the limiting distribution of the
ISBE in Theorem~\ref{theo:as.distrGS} is completely the same the one for the smoothed MLE and smoothed
Grenander-type estimator, as provided by Theorem~3.5 in~\cite{LopuhaaMustaSI2016}.
The following theorem shows that $\tilde{\lambda}_n^{GS}(x)$ is in fact asymptotically equivalent
to both these estimators.
In particular, this means that the order of smoothing and isotonization for the Grenander-type estimator
yields exactly the same limit behavior.
This is in line with the findings in~\cite{mammen1991} and~\cite{vdvaart-vdlaan2003}.
In order to use results from~\cite{LopuhaaMustaSI2016}, we have to strengthen condition (A2) slightly.
\begin{theo}
\label{theo:ISBE asymptotic equivalence}
Suppose that (A1) holds and
\begin{itemize}
\item[(A2')]
there exists $\epsilon>0$ such that
\[
\sup_{|\beta-\beta_0|\leq\epsilon}\E
\left[
|Z|^2
\left(
\mathrm{e}^{2\beta'Z}+\mathrm{e}^{4\beta'Z}
\right)
\right]<\infty.
\]
\end{itemize}%
Fix $x\in(0,\tau_h)$ and $\tau^*\in(x,\tau_H)$.
Assume that $\lambda_0$ is $m\geq2$ times continuously differentiable, with $\lambda'_0$ uniformly bounded from below by a strictly positive constant.
Assume that $t\mapsto \Phi(t;\beta_0)$ is differentiable with a bounded derivative in a neighborhood of $x$
and let~$k$ be $m$-orthogonal satisfying~\eqref{def:kernel}.
Let~$\tilde{\lambda}_n^{GS}$ be the left derivative of the greatest convex minorant on $[0,\tau^*]$ of $\Lambda_n^s$ defined in~\eqref{def:smoothed Breslow}
and suppose that $n^{1/(2m+1)}b\to c>0$.
Let $\tilde\lambda_n^{SG}$ be the smoothed Grenander-type estimator defined in~\cite{LopuhaaMustaSI2016}.
Then
\[
n^{m/(2m+1)}
\left(
\hat\lambda_n^{GS}(x)-\tilde\lambda_n^{SG}(x)
\right)
\to 0,
\]
in probability, and similarly if we replace $\tilde\lambda_n^{SG}(x)$
by the smoothed maximum likelihood estimator $\hat\lambda_n^{SM}(x)$, defined in~\cite{LopuhaaMustaSI2016}.
This also holds if we replace~$\tilde{\lambda}^{GS}_n(x)$ with $\tilde{\lambda}_n^{\mathrm{naive}}(x)$, defined in~\eqref{def:naive est GS}.
\end{theo}
The proof of Theorem~\ref{theo:ISBE asymptotic equivalence} can be found in Appendix~\ref{subsec:proofs GS}.

\section{Numerical results for pointwise confidence intervals}
\label{sec:conf-int}

In this section we illustrate the finite sample performance of the two estimators considered
in Sections~\ref{sec:MSLE} and~\ref{sec:GS} by constructing pointwise confidence intervals for the baseline hazard rate.
From Theorems~\ref{theo:distrMS} and~\ref{theo:as.distrGS},
it can be seen that the asymptotic $100(1-\alpha)\%$-confidence intervals at the point $x_0\in(0,\tau_H)$
are of the form
$\widehat{\lambda}_n^{IS}(x_0)
 -n^{-2/5}
(\widehat{\mu}_n(x_0)\pm\widehat{\sigma}_n(x_0)q_{1-\alpha/2})$,
where $q_{1-\alpha/2}$ is the $(1-\alpha/2)$ quantile of the standard normal distribution,
$\widehat{\lambda}_n^{IS}(x_0)$ is the isotonized smooth estimator at hand (either MSLE or ISBE),
and $\widehat{\mu}_n(x_0)$, $\widehat{\sigma}_n(x_0)$  are corresponding plug-in estimators of the asymptotic mean and standard deviation, respectively.
However, from the expression of the asymptotic mean in Theorems~\ref{theo:distrMS} and~\ref{theo:as.distrGS} for $m=2$,
it is obvious that obtaining the plug-in estimators requires estimation of second derivatives of~$\lambda_0$, $\Phi$ and~$h$.
Since accurate estimation of derivatives is a hard problem, we choose to avoid it by using undersmoothing.
This procedure is shown to be preferred to bias estimation,
because it is computationally more convenient and leads to better results
(see also~\cite{Hall92}, \cite{GJ15}, \cite{CHT06}).
Undersmoothing consists of using a bandwidth of a smaller order than the optimal one (in our case~$n^{-1/5}$).
As a result, the bias of $n^{2/5}( \widehat\lambda_n^{IS}(x_0)-\lambda_0(x_0))$, which is of the order $n^{2/5}b^2$ (see~\eqref{eqn:asymptotic_mean}), will converge to zero. On the other hand, the asymptotic variance is $n^{-1/5}b^{-1}\sigma^2$ (see~\eqref{eq:EYn2^2 GS} with $m=2$).
For example, with $b=n^{-1/4}$, asymptotically $n^{2/5}( \widehat\lambda_n^{IS}(x_0)-\lambda_0(x_0))$ behaves like a normal distribution with mean $n^{-1/10}$ and variance $n^{1/20}\sigma^2$. Hence, the confidence interval becomes
\begin{equation}
\label{def:confint2}
\widehat\lambda_n^{IS}(x_0)
\pm n^{-3/8}
\widehat{\sigma}_n(x_0)q_{1-\alpha/2},
\end{equation}
where
\[
\widehat{\sigma}_n(x_0)=\frac{\widehat{\lambda}_n^{IS}(x_0)}{c \Phi_n(x_0;\hat{\beta}_n)}\int_{-1}^1k(y)^2\,\mathrm{d} y.
\]
In our simulations, the event times are generated from a Weibull baseline distribution with shape parameter $1.5$ and scale
parameter $1$.
The real valued covariate and the censoring time are chosen to be uniformly distributed on the interval $(0,1)$ and we take $\beta_0=0.5$.
Confidence intervals are calculated at the point $x_0=0.5$ using 1000 sets of data.
We take bandwidth $b=cn^{-1/4}$, with $c=1$, and kernel function $k(u)=(35/32)(1-u^2)^3\1_{\{|u|\leq 1\}}$.

\begin{table}[t]
\begin{tabular}{ccccccccccc}
\toprule
      &&    \multicolumn{2}{c}{MSLE}&   \multicolumn{2}{c}{ISBE} &&    \multicolumn{2}{c}{MSLE$^*$}&   \multicolumn{2}{c}{ISBE}$^*$ \\
$n$     && AL    & CP & AL    & CP    && AL    & CP & AL    & CP  \\
100   &&  1.035 & 0.779 & 1.007 & 0.740 && 1.050 & 0.970 & 0.981 & 0.911 \\
500   &&  0.562 & 0.840 & 0.554 & 0.841 && 0.556 & 0.979 & 0.547 & 0.940 \\
1000  &&  0.424 & 0.846 & 0.428 & 0.843 && 0.429 & 0.976 & 0.424 & 0.953\\
5000  &&  0.234 & 0.892  & 0.234 & 0.888 && 0.234 & 0.975 & 0.248 & 0.960\\
\bottomrule
\end{tabular}
\caption{The average length (AL) and the coverage probabilities (CP) for $95\%$ pointwise confidence intervals of the baseline hazard rate at the point $x_0=0.5$ based on the asymptotic distribution.
MSLE and ISBE use $\hat{\beta}_n$, while MSLE$^*$ and ISBE$^*$ use $\beta_0$.}
\label{tab:1}
\end{table}
Table~\ref{tab:1} shows the performance of the estimators.
The four columns corresponding to MSLE and ISBE list the average length (AL) and the coverage probabilities (CP)
of the confidence intervals given in~\eqref{def:confint2} for various sample sizes.
Clearly, the coverage probabilities are far from the nominal level of $95\%$.
However, smoothing does lead to significantly better results in comparison with the non-smoothed estimators
(e.g., the traditional Grenander-type estimator considered in~\cite{LopuhaaNane2013} has coverage probabilities $0.468$, $0.479$, $0.593$, $0.642$ and $0.783$).
The performance depends strongly on the choice of the constant $c$ in the bandwidth $b=cn^{-1/4}$,
because the asymptotic length is inversely proportional to $c$ (see~\eqref{def:confint2}).
This means that, by choosing a smaller $c$, one will get wider confidence intervals and higher coverage probabilities.
Unfortunately, it is not clear what would be the optimal choice of such a constant.
This is a common problem in the literature (e.g., see~\cite{CHT06} and \cite{GCM96}).
As indicated in~\cite{mullerwang1990}, cross-validation methods that consider a trade-off between bias and variance
suffer from the fact that the variance of the estimator increases as one approaches the endpoint of the support.
This is even enforced in our setting, because the bias is also decreasing when approaching the endpoint of the support.
We tried a locally adaptive choice of the bandwidth, as proposed in~\cite{MW90},
by minimizing an estimator of the Mean Squared Error, but in our setting this method did not lead to better results.
A simple choice is to take $c$ equal to the range of the date (see~\cite{GJ15}), which in our case corresponds to $c=1$.

More importantly, estimation of the parameter $\beta_0$ has a greater effect on the accuracy of the results.
The last four columns of Table~\ref{tab:1} show that if we use the true value of $\beta_0$ in the computation of the estimators , the coverage probabilities improve significantly and the~ISBE seems to perform better.
\begin{figure}[t]
\includegraphics[ width=0.45\textwidth]{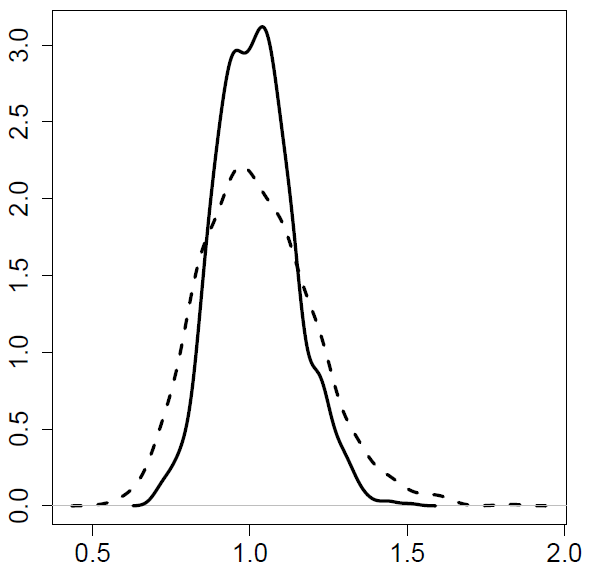}
\quad
\includegraphics[ width=0.45\textwidth]{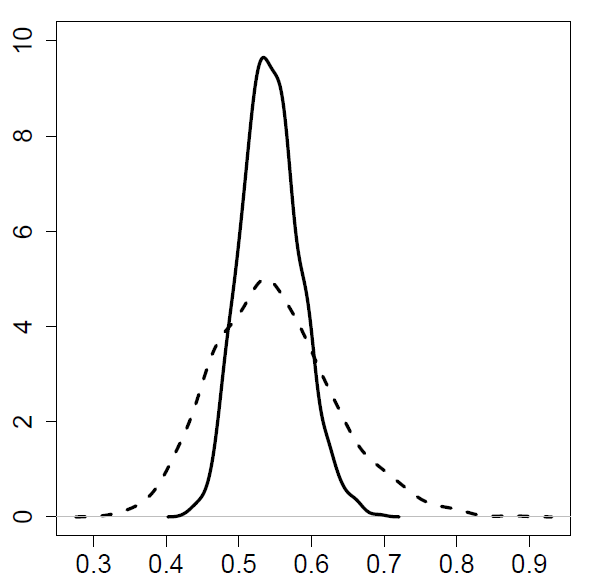}
\caption{Left panel: Values of the ISBE computed using the true parameter~$\beta_0$ (solid line)
and the Cox's partial MLE~$\hat{\beta}_n$ (dashed line).
Right panel: Values of the length of the confidence interval computed using the true parameter $\beta_0$ (solid line)
and the Cox's partial MLE~$\hat{\beta}_n$ (dashed line).}
\label{fig:est_length_beta_0}
\end{figure}
Things are illustrated in Figure~\ref{fig:est_length_beta_0},
which shows the kernel densities of the values of the ISBE and the corresponding lengths of the confidence intervals,
computed using the true parameter $\beta_0$ and the partial ML estimator $\hat{\beta}_n$, for $1000$ samples of size $n=500$.
We conclude that the use of $\hat{\beta}_n$ leads to underestimation or overestimation of both $\lambda_0(x_0)$
as well as the corresponding length of the confidence interval.
In fact, underestimation of both goes hand in hand, since the variance of the ISBE is proportional to $\lambda_0(x_0)$,
and similarly for overestimation.
As can be seen in Table~\ref{tab:1}, estimation of $\beta_0$ does not seem to effect the length of the confidence interval.
However, the coverage probabilities change significantly.
When $\lambda_0(x_0)$ is underestimated, the midpoint of the confidence interval lies below $\lambda_0(x_0)$
and the simultaneous underestimation of the length even stronger prevents the confidence interval to cover $\lambda_0(x_0)$.
When $\lambda_0(x_0)$ is overestimated, the midpoint of the confidence interval lies above $\lambda_0(x_0)$,
but the simultaneous overestimation of the length does not compensate this,
so that the confidence interval to often fails to cover $\lambda_0(x_0)$.

Although the partial ML estimator $\hat{\beta}_n$ is a standard estimator for the regression coefficients,
the efficiency results are only asymptotic.
As pointed out in~\cite{CO84}  and ~\cite{RZ11}, for finite samples the use of the partial likelihood leads to a loss of accuracy.
Recently, \cite{RZ11} introduced
the~MLE for $\beta_0$ obtained by joint maximization of  the loglikelihood in~\eqref{def:loglikelihood cox} over both $\beta$ and $\lambda_0$.
It was shown that for small and moderate sample sizes, the joint MLE for $\beta_0$ performs better than~$\hat{\beta}_n$.
However, in our case, using this estimator instead of $\hat{\beta}_n$, does not bring any essential difference in the coverage probabilities.
Finally, we notice that the performance of the MSLE and ISBE are comparable.

\begin{figure}[t]
\centering
\subfloat[][MSLE]
{\includegraphics[width=0.45\textwidth]{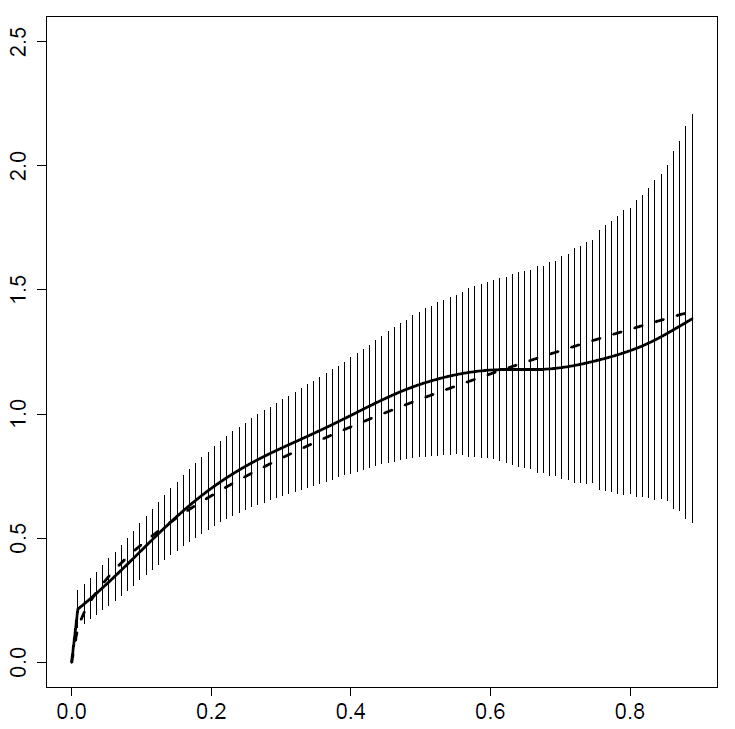}} \quad
\subfloat[][ISBE]
{\includegraphics[width=0.45\textwidth]{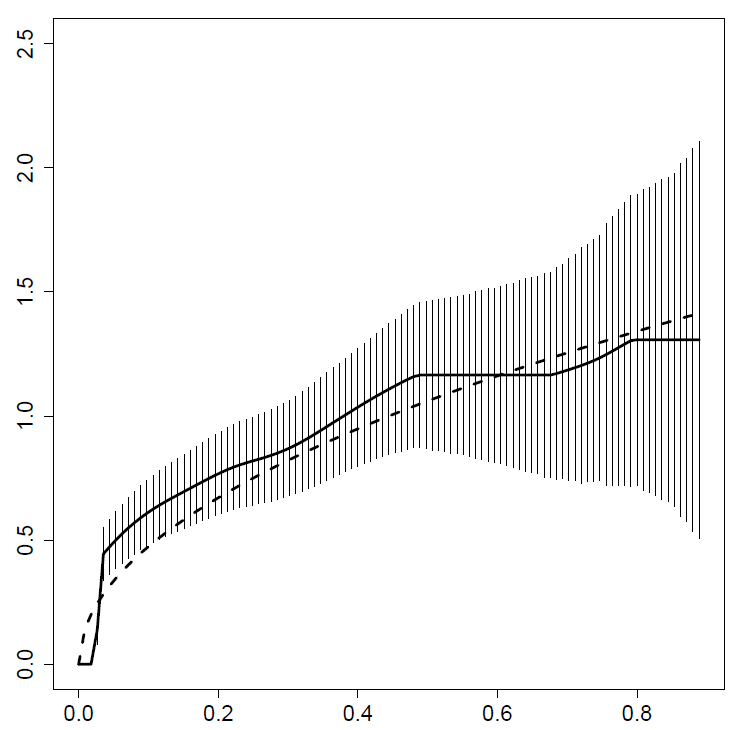}}
\caption{$95\%$ confidence intervals based on the asymptotic distribution for the baseline hazard rate using undersmoothing.}
\label{fig:subfig4}
\end{figure}
The behavior of the two methods for a fixed sample size~$n=500$ at different points of the support is illustrated in Figure~\ref{fig:subfig4}.
The results are again comparable and the common feature is that the length increases as we move to the left boundary.
This is due to the fact that the length is proportional to the asymptotic standard deviation, which in this case turns out to be increasing,
$\sigma^2(x)=1.5 \sqrt{x}/(c\Phi(x;\beta_0))$.
Note that $\Phi(x;\beta_0)$ defined in~\eqref{eq:def Phi} is decreasing.

An alternative to confidence intervals based on the asymptotic distribution relies on the bootstrap.
Studies on bootstrap confidence intervals in the Cox model are investigated in~\cite{Burr94} and~\cite{XSY2014}.
We follow one of their proposals for a smooth bootstrap.
We fix the covariates and we generate the event time $X_i^*$ from a smooth estimate for the cdf of~$X$ conditional on $Z_i$:
\[
\hat{F}_n(x|Z_i)
=
1-\exp\left\{-\Lambda^s_n(x)\mathrm{e}^{\hat{\beta}'_nZ_i} \right\},
\]
where $\Lambda^s_n$ is the smoothed Breslow estimator defined in~\eqref{def:smoothed Breslow}.
The censoring times~$C^*_i$ are generated from the Kaplan-Meier estimate $\hat{G}_n$.
Then we take~$T^*_i=\min(X^*_i,C^*_i)$ and $\Delta^*_i=\1_{\{X^*_i\leq C^*_i\}}$.
For constructing the confidence intervals, we used $1000$ bootstrap samples $(T^*_i,\Delta^*_i,Z_i)$,
and for each boostrap sample we computed~$\tilde{\lambda}^{MS,*}_n(x_0)$ and $\tilde{\lambda}^{GS,*}_n(x_0)$.
Here the kernel function is the same as before and the bandwidth is taken to be $b=n^{-1/5}$.
Then, the $100(1-\alpha)\%$ confidence interval is given by
\[
\left[
q^*_{\alpha/2}(x_0),q^*_{1-\alpha/2}(x_0)
\right],
\]
where $q^*_{\alpha}(x_0)$ is the $\alpha$-percentile of the $1000$ values of the estimates $\tilde{\lambda}^{MS,*}_n(x_0)$ or $\tilde{\lambda}^{GS,*}_n(x_0)$.

We investigate the behavior of the two estimators in the following two different settings
\[
\begin{tabular}{ccc}
 \toprule
     & \textbf{Model 1}   & \textbf{Model 2}   \\
\\[-10pt]
$X$   & Weibull (1.5,1)  & Weibull (3,1)  \\
$C$  & Uniform (0,1)  & Uniform (0,2)   \\
$Z$  & Uniform (0,1) & Bernoulli (0.5)   \\
$\beta_0$ & 0.5 & 0.1\\
$x_0$ & 0.5 & 0.5\\
\bottomrule\\
\end{tabular}
\]
Note that Model 1 is the same as in the previous simulation. The main differences between the two models are the following:
the baseline hazard rate is slightly increasing in model 1 and strongly increasing in model 2,
the covariates have a smaller effect on the hazard rate in model 2,
and model 1 corresponds to $35\%$ uncensored observations, while in the model 2 we have about $50\%$ uncensored observations.
It is also worthy noticing that, for model 1, we calculate the confidence intervals at the middle point of the support $x_0=0.5$
in order to avoid boundary problems, while, in model 2 we again consider $x_0=0.5$,
because the estimation becomes more problematic on the interval $[1,2]$.
This is probably due to the fact that we only have a few observations in this time interval,
on which the hazard rate is strongly increasing.

The average length and the empirical coverage for $1000$ iterations and different sample sizes are reported in
Table~\ref{tab:bootstrap1}.
\begin{table}[h]
\begin{tabular}{ccccccccccc}
\toprule
&&    \multicolumn{4}{c}{\textbf{Model 1}} &&    \multicolumn{4}{c}{\textbf{Model 2}} \\
\\[-10pt]
\cline{3-6} \cline{8-11}
\\[-10pt]
&&    \multicolumn{2}{c}{MSLE}&   \multicolumn{2}{c}{ISBE} &&    \multicolumn{2}{c}{MSLE}&   \multicolumn{2}{c}{ISBE} \\
$n$     && AL    & CP   & AL &CP  && AL    & CP   & AL & CP \\
100   && 1.553 & 0.943  & 1.625 & 0.914 && 0.766 & 0.970 & 0.858 & 0.975 \\
500   && 0.701 & 0.947 & 0.726 & 0.941  && 0.362 & 0.959 & 0.395 & 0.951 \\
1000  &&0.512& 0.949  & 0.527 & 0.963 && 0.271 & 0.959 & 0.286 & 0.962 \\
\bottomrule\\
\end{tabular}
\caption{The average length (AL) and the coverage probabilities (CP) for the $95\%$ bootstrap confidence intervals
of the baseline hazard rate at the point $x_0=0.5$ for Model 1 and 2.}
\label{tab:bootstrap1}
\end{table}
We observe that bootstrap confidence intervals behave better that confidence intervals constructed on the basis of the asymptotic distribution,
i.e., the coverage probabilities are closer to the nominal level of $95\%$.
Results also indicate that the MSLE behaves slightly better than the ISBE since, in general,
it leads to shorter confidence intervals and better coverage probabilities.

\appendix
\section{Proofs}
\label{sec:appendix}
\subsection{Proofs for Section~\ref{sec:MSLE}}
\label{subsec:proofs MSLE}

\begin{proof}[Proof of Lemma~\ref{lem:char MSLE}]
We start by writing
\[
\begin{split}
\ell^s_\beta(\lambda_0)
&=
\frac{1}{n}\sum_{i=1}^n
\left\{
\Delta_i\int_0^\infty\log \lambda_0(t)k_b(t-T_i)\,\mathrm{d}t
-
\mathrm{e}^{\beta'Z_i}\int_0^\infty
\left(
\int_0^t \lambda_0(u)\,\mathrm{d}u
\right)k_b(t-T_i)
\,\mathrm{d}t\right\}\\
&=
\int_0^\infty\log \lambda_0(t)
\left( \frac{1}{n}\sum_{i=1}^n \Delta_i
k_b(t-T_i)\right)\,\mathrm{d}t\\
&\qquad\qquad\qquad\qquad\qquad\qquad-
\int_0^\infty
\lambda_0(u)
\left(\frac{1}{n}\sum_{i=1}^n\mathrm{e}^{\beta'Z_i}
\int_{u}^\infty k_b(t-T_i)\,\mathrm{d}t\right)\,\mathrm{d}u,
\end{split}
\]
which is equal to
\[
\begin{split}
\int_0^\infty
\Big\{
v_n(t)\log\lambda_0(t)-w_n(t;\beta)\lambda_0(t)
\Big\}\,\mathrm{d}t
=
\int_0^\infty
\left\{
\frac{v_n(t)}{w_n(t;\beta)}\log\lambda_0(t)-\lambda_0(t)
\right\}w_n(t;\beta)\,\mathrm{d}t,
\end{split}
\]
with $v_n$ and $w_n$ defined in~\eqref{eqn:v_n w_n}.
Maximizing the right hand side over nondecreasing $\lambda_0$ is equivalent to minimizing
\begin{equation}
\label{eq:DeltaPhi}
\int
\Delta_\Phi
\left(
\frac{v_n(t)}{w_n(t;\beta)},\lambda(t)
\right)
w_n(t;\beta)\,\mathrm{d}t
\end{equation}
over nondecreasing $\lambda$, where
$\Delta_\Phi(u,v)
=
\Phi(u)-\Phi(v)-(u-v)\phi(v)$,
with $\Phi(u)=u\log u$.
Theorem~1 in~\cite{GJ10} provides a characterization of the minimizer~$\hat{\lambda}^s_n(x;\beta)$ of~\eqref{eq:DeltaPhi},
and hence of the maximizer of $\ell^s_\beta$.
It is the unique solution of a generalized continuous isotonic regression problem, i.e.,
it is continuous and it is the minimizer of
\[
\psi(\lambda)=\frac{1}{2}\int \left(\lambda(x)-\frac{v_n(x)}{w_n(x;\beta)}\right)^2
w_n(x;\beta)\,\mathrm{d}x,
\]
over all nondecreasing functions $\lambda$ and can be described as the slope of the GCM of the graph defined by~\eqref{eqn:graph}.
\end{proof}

\begin{proof}[Proof of Lemma~\ref{lem:coincide}]
It is enough to prove that for an arbitrarily fixed $\epsilon>0$ and for sufficiently large $n$
\[
\p
\left(
{\lambda}_n^{\mathrm{naive}}(t)={\lambda}_n^{IS}(t),\text{ for all } t\in[\ell,M]
\right)
\geq 1-\epsilon.
\]
Recall that $\lambda^{IS}_n(t)$ is defined as the slope of the greatest convex minorant
$\big\{\big(X_n(t),\widehat Y_n(t)\big), t\in[0,\hat\tau]\big\}$ of the graph
$\big\{\big( X_n(t),Y_n(t))\big),t\in[0,\hat\tau]\big\}$.
We consider $Y_n$ on the interval $[\ell,M]$ and define the linearly extended  version of~$Y_n$ on $[0,\hat\tau]$ by
\[
Y^*_n(t)=
\begin{cases}
Y_n(\ell)+\big( X_n(t)- X_n(\ell)\big){\lambda}_n^{\mathrm{naive}}(\ell), &\text{ for } t\in[0,\ell),\\
Y_n(t), &\text{ for } t\in[\ell,M],\\
Y_n(M)+\big( X_n(t)- X_n(M)\big){\lambda}_n^{\mathrm{naive}}(M), &\text{ for } t\in(M,\hat\tau].
\end{cases}
\]
It suffices to prove that, for sufficiently large $n$,
\begin{equation}
\label{eq:cor prop1}
\p
\left(
\left\{
\big( X_n(t),Y^*_n(t)\big):t\in[0,\hat\tau]
\right\} \text{ is convex }
\right)\geq 1-\epsilon/2,
\end{equation}
and
\begin{equation}
\label{eq:cor prop2}
\p
\left(
Y^*_n(t)\leq Y_n(t),
\text{ for all }t\in [0,\hat\tau]
\right)\geq 1-\epsilon/2.
\end{equation}
Indeed, if~\eqref{eq:cor prop1} and~\eqref{eq:cor prop2} hold,
then with probability greater than or equal to $1-\epsilon$,
the curve $\big\{\big( X_n(t),Y^*_n(t)\big):t\in[0,\hat\tau]\big\}$ is a convex curve lying below the graph
$\big\{\big( X_n(t),Y_n(t)\big):t\in[0,\hat\tau]\big\}$, with $Y^*_n(t)=Y_n(t)$ for all $t\in[\ell,M]$.
Hence, $Y_n(t)=Y^*_n(t)\leq \widehat Y_n(t)\leq Y_n(t)$,
for all $t\in[\ell,M]$.
It follows that, for sufficiently large $n$,
\[
\p\left(
{\lambda}_n^{\mathrm{naive}}(t)=\frac{\mathrm{d}Y_n(t)}{\mathrm{d} X_n(t)}=\frac{\mathrm{d}\widehat Y_n(t)}{\mathrm{d} X_n(t)}={\lambda}_n^{IS}(t),
\text{ for all }t\in [\ell,M]
\right)\geq 1-\epsilon.
\]

To prove~\eqref{eq:cor prop1}, define the event
\[
A_n=\left\{{\lambda}_n^{\mathrm{naive}} \text{ is increasing on } [\ell-\eta_1,M+\eta_2]\right\},
\]
for $\eta_1\in(0,\ell)$ and $\eta_2\in(0,\hat\tau-M)$.
Note that on the intervals $[0,\ell)$ and $(M,\hat\tau]$,
the curve $\big\{\big( X_n(t),Y^*_n(t)\big):t\in[0,\hat\tau]\big\}$ is the tangent line of the graph $\big\{\big( X_n(t),Y_n(t)\big):t\in[0,\hat\tau]\big\}$
at the points $\big( X_n(\ell),Y_n(\ell)\big)$ and $\big( X_n(M),Y_n(M)\big)$.
As a result, on the event $A_n$ the curve is convex, so that together with condition~\eqref{cond c},
for sufficiently large $n$
\[
\p
\left(
\left\{
\big( X_n(t),Y^*_n(t)\big):t\in[0,\hat\tau]
\right\} \text{ is convex }
\right)
\geq
\p(A_n)
\geq
1-\epsilon/2.
\]
To prove~\eqref{eq:cor prop2}, we split the interval $[0,\hat\tau]$ in five different intervals
$I_1=[0,\ell-\eta_1),$ $I_2=[\ell-\eta_1,\ell),$ $I_3=[\ell,M],$ $I_4=(M,M+\eta_2]$, and $I_5=(M+\eta_2,\hat\tau]$,
and show that
\begin{equation}
\label{eqn:C}
\p
\left(
Y^*_n(t)\leq Y_n(t),\text{ for all }t\in I_i
\right)\geq 1-\epsilon/10,
\quad
i=1,\ldots,5.
\end{equation}
For $t\in I_3,$ $Y^*_n(t)=Y_n(t)$ and thus~\eqref{eqn:C} is trivial.
For $t\in I_2$, by the mean value theorem,
\[
Y_n(t)-Y_n(\ell)=\big( X_n(t)- X_n(\ell)\big)\,{\lambda}_n^{\mathrm{naive}}(\xi_t),
\]
for some $\xi_t\in[t,\ell]$.
Thus, since $ X_n(t)\leq X_n(\ell)$ according to condition~\eqref{cond a},
\begin{equation}
\label{eq:argument I2}
\begin{split}
&
\p
\left(
Y^*_n(t)\leq Y_n(t),\text{ for all }t\in I_2
\right)
\\
&=
\p
\left(
\big( X_n(t)- X_n(\ell)\big)
\big({\lambda}_n^{\mathrm{naive}}(\xi_t)-{\lambda}_n^{\mathrm{naive}}(\ell)\big)\geq 0,
\text{ for all }t\in I_2
\right)\\
&=
\p
\left(
\big({\lambda}_n^{\mathrm{naive}}(\xi_t)-{\lambda}_n^{\mathrm{naive}}(\ell)\big)\leq 0,
\text{ for all }t\in I_2\right)\\
&\geq
\p(A_n)\geq 1-\epsilon/10,
\end{split}
\end{equation}
for $n$ sufficiently large, according to condition~\eqref{cond c}.
The argument for $t\in I_4$ is exactly the same.
Furthermore, making use of condition~\eqref{cond d}, for each $t\in I_1,$ we obtain
\[
\begin{split}
Y_0(t)-Y_0(\ell)-\lambda_0(\ell)\big(X_0(t)-X_0(\ell)\big)
&=
\int_t^{\ell} \big(\lambda_0(\ell)-\lambda_0(u)\big)\,\mathrm{d}X_0(u)\\
&\geq
\int_{\ell-\eta_1}^{\ell}
\big(\lambda_0(\ell)-\lambda_0(u)\big)\,\mathrm{d}X_0(u).
\end{split}
\]
This implies that
\[
\begin{split}
Y_n(t)-Y^*_n(t)
&=
Y_n(t)-
Y_n(\ell)-\big( X_n(t)- X_n(\ell)\big){\lambda}_n^{\mathrm{naive}}(\ell)\\
&\geq
Y_n(t)-Y_0(t)+Y_0(\ell)-Y_n(\ell)+\lambda_0(\ell)\big(X_0(t)- X_n(t)\big)\\
&\quad+
\lambda_0(\ell)\big( X_n(\ell)-X_0(\ell)\big)
+
\big({\lambda}_n^{\mathrm{naive}}(\ell)-\lambda_0(\ell)\big)\big( X_n(\ell)- X_n(t)\big)\\
&\qquad+
\int_{\ell-\eta_1}^{\ell}\big(\lambda_0(\ell)-\lambda_0(u)\big)\mathrm{d}X_0(u)\\
&\geq
-2\sup_{t\in[0,\ell]}|Y_n(t)-Y_0(t)|
-2\lambda_0(\ell)\sup_{t\in[0,\ell]}| X_n(t)-X_0(t)|\\
&\qquad-
2|{\lambda}_n^{\mathrm{naive}}(\ell)-\lambda_0(\ell)|\sup_{t\in[0,\ell]}| X_n(t)|
+
\int_{\ell-\eta_1}^{\ell}\big(\lambda_0(\ell)-\lambda_0(u)\big)\mathrm{d}X_0(u).
\end{split}
\]
According to conditions~(b) and~(c), the first three terms on the right hand side tend to zero in probability.
This means that the probability on the left hand side of~\eqref{eqn:C} for $i=1$, is bounded
from below by
\[
\p
\left(
Z_n\leq \int_{\ell-\eta_1}^\ell\big(\lambda_0(\ell)-\lambda_0(u)\big)\mathrm{d}X_0(u)
\right),
\]
where $Z_n=o_p(1)$.
This probability is greater than $1-\epsilon/10$ for $n$ sufficiently large, since
\[
\int_{\ell-\eta_1}^\ell\big(\lambda_0(\ell)-\lambda_0(u)\big)\,\mathrm{d}X_0(u)
\geq
x_0(\ell)
\int_{\ell-\eta_1}^\ell\big(\lambda_0(\ell)-\lambda_0(u)\big)\,\mathrm{d}u
>0,
\]
using that $\lambda_0$ is strictly increasing.
For $I_5$ we can argue exactly in the same way.
\end{proof}

\begin{lemma}
\label{le:1}
Suppose that (A1)-(A2) hold.
Let $H^{uc}(t)$ and $\Phi(t;\beta_0)$ be defined in~\eqref{eq:def Huc} and~\eqref{eq:def Phi}, and let $h(t)=\mathrm{d}H^{uc}(t)/\mathrm{d}t$.
Suppose that $h$ and $t\mapsto\Phi(t;\beta_0)$ are $m\geq 1$ times continuously differentiable
and let $k$ be $m$-orthogonal satisfying~\eqref{def:kernel}.
Then, for each $0<\ell<M<\tau_H$, it holds
\begin{equation}
\label{eqn:lemma1-1}
\begin{split}
\sup_{t\in[\ell,M]}|v_n(t)-h(t)|
&=
O(b^m)+O_p(b^{-1}n^{-1/2}),\\
\sup_{t\in[\ell,M]}|w_n(t;\hat{\beta}_n)-\Phi(t;\beta_0)|
&=
O(b^m)+O_p(b^{-1}n^{-1/2}),
\end{split}
\end{equation}
where $v_n$, $w_n$ and $\Phi$ are defined in~\eqref{eqn:v_n w_n} and~\eqref{eq:def Phi}.
\end{lemma}

\begin{proof}
To obtain the first result in~\eqref{eqn:lemma1-1}, we write
\[
v_n(t)-h(t)=v_n(t)-h_s(t)+h_s(t)-h(t),
\]
where
\begin{equation}
\label{def:hs}
h_s(t)=\int k_b(t-u)\,h(u)\,\mathrm{d}u.
\end{equation}
By a change of variable and a Taylor expansion, using that $k$ is $m$-orthogonal,
we deduce that
\begin{equation}
\label{eq:taylor h}
\begin{split}
h_s(t)-h(t)
&=
\int k_b(t-u)h(u)\,\mathrm{d}u-h(t)
=
\int_{-1}^1 k(y)
\left\{
h(t-by)-h(t)
\right\}\,\mathrm{d}y\\
&=
\int_{-1}^1 k(y)
\left\{
-h'(t)by+\cdots+\frac{h^{(m-1)}(t)}{(m-1)!}(-by)^{m-1}+\frac{h^{(m)}(\xi_{ty})}{m!}(-by)^m
\right\}\,\mathrm{d}y\\
&=
\frac{(-b)^m}{m!}
\int_{-1}^1
h^{(m)}(\xi_{ty})k(y)y^m\,\mathrm{d}y,
\end{split}
\end{equation}
for some $|\xi_{ty}-t|<|by|$.
It follows that
\begin{equation}
\label{eq:h term}
\sup_{t\in[\ell,M]}|h_s(t)-h(t)|
\leq
\frac{b^m}{m!}
\sup_{t\in[0,\tau_H]}\left|h^{(m)}(t)\right|
\int_{-1}^1
|k(y)||y|^m\,\mathrm{d}y
=
O(b^m).
\end{equation}
Let $H^{uc}_n$ be the empirical sub-distribution function of the uncensored observations, defined by
\[
H_n^{uc}(x)
=
\int \delta \1_{\{t\leq x\}}\,\mathrm{d}\p_n(t,\delta,z).
\]
Then integration by parts yields
\begin{equation}
\label{eq:v minus h}
\begin{split}
v_n(t)-h_s(t)&=\int k_b(t-u)\,\mathrm{d}(H^{uc}_n-H^{uc})(u)\\
&=
-\int \frac{\partial}{\partial u}k_b(t-u)\,(H^{uc}_n-H^{uc})(u)\,\mathrm{d}u\\
&=
\frac{1}{b}\int_{-1}^1 k'(y)\,(H^{uc}_n-H^{uc})(t-by)\,\mathrm{d}y.
\end{split}
\end{equation}
Note that $H^{uc}_n(x)-H^{uc}(x)=\int \delta\1_{\{u\leq x\}}\,\mathrm{d}(\p_n-\p)(u,\delta,z)$.
Because the class of functions $\F=\left\{f(\cdot;x)\,:x\in[0,\tau_H]\right\}$,
with $f(u;x)=\1_{\{u\leq x\}}$, is a VC-class (e.g., see Example~2.6.1 in~\cite{VW96}),
also the class of functions $\{\mathcal{G}=\delta f:f\in \mathcal{F}\}$
is a VC-class, according to Lemma~2.6.18 in~\cite{VW96}.
It follows that the class $\mathcal{G}$ is Donsker, i.e., the process $\sqrt{n}(H^{uc}_n-H^{uc})$ converges weakly,
see Theorems~2.6.8 and~2.5.2 in~\cite{VW96}.
It follows by the continuous mapping theorem that
\begin{equation}
\label{eqn:H_uc}
\sqrt{n}\sup_{x\in[0,\tau_H]}|H^{uc}_n(x)-H^{uc}(x)|=O_p(1).
\end{equation}
Hence, we get
\begin{equation}
\label{eq:v minus h limit}
\sup_{t\in[\ell,M]}|v_n(t)-h_s(t)|
\leq
\frac{1}{b}
\sup_{x\in[\ell,M]}|H^{uc}_n(x)-H^{uc}(x)|
\int_{-1}^1 |k'(y)|\,\mathrm{d}y
=
O_p(b^{-1}n^{-1/2}).
\end{equation}
Together with~\eqref{eq:h term}, this proves the first result in~\eqref{eqn:lemma1-1}.

To prove the second result in~\eqref{eqn:lemma1-1}, note that from~\eqref{eqn:v_n w_n} and~\eqref{eq:def Phin} we have
\begin{equation}
\label{eqn:w}
w_n(t;\beta)
=
\frac{1}{n}\sum_{i=1}^n
\mathrm{e}^{\beta'Z_i}\int_{t}^\infty k_b(u-T_i)\,\mathrm{d}u
=
\int_{-1}^1 k(y)\Phi_n(t-by;\beta)\,\mathrm{d}y.
\end{equation}
Consequently, we can write
\[
\begin{split}
&
w_n(t;\hat{\beta}_n)-\Phi(t;\beta_0)\\
&=
\int_{-1}^1 k(y)\left\{\Phi_n(t-by;\hat{\beta}_n)-\Phi(t;\beta_0)\right\}\,\mathrm{d}y\\
&=
\int_{-1}^1 k(y)
\left\{
\Phi_n(t-by;\hat{\beta}_n)-\Phi(t-by;\beta_0)
\right\}
\,\mathrm{d}y
+
\int_{-1}^1 k(y)
\left\{
\Phi(t-by;\beta_0)-\Phi(t;\beta_0)
\right\}
\,\mathrm{d}y.
\end{split}
\]
Similar to~\eqref{eq:taylor h} and~\eqref{eq:h term}, for the second term on the right hand side, we obtain
\[
\sup_{t\in[\ell,M]}
\left|
\int_{-1}^1 k(y)
\left\{
\Phi(t-by;\beta_0)-\Phi(t;\beta_0)
\right\}
\,\mathrm{d}y
\right|
=
O(b^m).
\]
Hence, by means of the triangular inequality,
\[
\sup_{t\in[\ell,M]}
\left|w_n(t;\hat{\beta}_n)-\Phi(t;\beta_0)\right|
\leq
\sup_{x\in\R}
\left|\Phi_n(t;\hat{\beta}_n)-\Phi(t;\beta_0)\right|
+
O(b^m)
=
O_p(n^{-1/2})+O(b^m),
\]
according to Lemma~4 in~\cite{LopuhaaNane2013}.
\end{proof}

\begin{lemma}
\label{lem:cons-naive}
Let $\hat{\lambda}_n^{\mathrm{naive}}$ be defined in~\eqref{def:naive est MSLE}.
Then, under the assumptions of Lemma~\ref{le:1}, for each $0<\ell<M<\tau_H$,
\[
\sup_{x\in[\ell,M]}|\hat{\lambda}_n^{\mathrm{naive}}(x)-\lambda_0(x)|
=
O(b^{m})+O_p(b^{-1}n^{-1/2}).
\]
\end{lemma}
\begin{proof}
By~\eqref{eqn:lambda0} and the definition of $\hat{\lambda}_n^{\mathrm{naive}}$, we have
\[
\begin{split}
\sup_{x\in[\ell,M]}|\hat{\lambda}_n^{\mathrm{naive}}(x)-\lambda_0(x)|
&=
\sup_{x\in[\ell,M]}\left|\frac{v_n(x)}{w_n(x;\hat{\beta}_n)}-\frac{h(x)}{\Phi(x;\beta_0)}\right|\\
&\leq
\frac{\sup_{x\in[\ell,M]}\left|v_n(x)\Phi(x;\beta_0)-h(x)w_n(x;\hat{\beta}_n)\right|}{|w_n(M;\hat{\beta}_n)|\Phi(M;\beta_0)}.
\end{split}
\]
The triangular inequality and Lemma~\ref{le:1} yield
\[
\sup_{x\in[\ell,M]}\left|v_n(x)\Phi(x;\beta_0)-h(x)w_n(x;\hat{\beta}_n)\right|=O(b^{m})+O_p(b^{-1}n^{-1/2}).
\]
and $w_n(M;\hat{\beta}_n)^{-1}=O_p(1).$
The statement follows immediately.
\end{proof}

\begin{lemma}
\label{le:1a}
Suppose that (A1)-(A2) hold.
Let $H^{uc}(t)$ and $\Phi(t;\beta_0)$ be defined in~\eqref{eq:def Huc} and~\eqref{eq:def Phi}, and let $h(t)=\mathrm{d}H^{uc}(t)/\mathrm{d}t$.
Suppose that $h$ and $t\mapsto\Phi(t;\beta_0)$ are $m\geq 1$ times continuously differentiable
and let $k$ be $m$-orthogonal satisfying~\eqref{def:kernel}.
If $b\to0$ and $1/b=O(n^{\alpha})$, for some $\alpha\in(0,1/4)$, then
for each $0<\ell<M<\tau_H$, it holds
\begin{equation}
\label{eqn:lemma1-2}
\sup_{t\in[\ell,M]}|v'_n(t)-h'(t)|\xrightarrow{\p}0,
\qquad
\sup_{t\in[\ell,M]}|w'_n(t;\hat{\beta}_n)-\Phi'(t;\beta_0)|\xrightarrow{\p}0,
\end{equation}
where $v_n$, $w_n$ and $\Phi$ are defined in~\eqref{eqn:v_n w_n} and~\eqref{eq:def Phi}.
\end{lemma}
\begin{proof}
Let us consider the first statement  of~\eqref{eqn:lemma1-2}.
We write
\[
v'_n(t)-h'(t)
=
v'_n(t)-h'_s(t)+h'_s(t)-h'(t),
\]
where $h_s$ is defined in~\eqref{def:hs}.
For the second term we have
\[
\sup_{t\in[\ell,M]}
\left|h'_s(t)-h'(t)\right|
\leq
\sup_{t\in[\ell,M]}
\int_{-1}^1
|k(y)|\left|h'(t-by)-h'(t)\right|\,\mathrm{d}y\to0,
\]
by the uniform continuity of $h'$.
Moreover, similar to~\eqref{eq:v minus h} and~\eqref{eq:v minus h limit},
\[
\sup_{t\in[\ell,M]}
\left|v'_n(t)-h'_s(t)\right|
\leq
\frac{1}{b^2}
\sup_{x\in[\ell,M]}|H^{uc}_n(x)-H^{uc}(x)|
\int_{-1}^1 |k''(y)|\,\mathrm{d}y
=
O_p(n^{2\alpha-1/2}),
\]
which tends to zero in probability, as $\alpha<1/4$.
To obtain the second statement of~\eqref{eqn:lemma1-2}, first note that from~\eqref{eqn:v_n w_n},
\begin{equation}
\label{eqn:w'_n}
w'_n(t;\hat{\beta}_n)
=
\int k'_b(t-u)\Phi_n(u;\hat{\beta}_n)\,\mathrm{d}u,
\end{equation}
and write
\begin{equation}
\label{eq:decomp wn'}
w'_n(t;\hat{\beta}_n)-\Phi'(t;\beta_0)
=
w'_n(t;\hat{\beta}_n)-w'_s(t;\beta_0)
+
w'_s(t;\beta_0)-\Phi'(t;\beta_0),
\end{equation}
where
\[
w_s(t;\beta_0)=\int k_b(t-u)\Phi(u;\beta_0)\,\mathrm{d}u.
\]
For the second difference on the right hand side of~\eqref{eq:decomp wn'} we have
\begin{equation}
\label{eqn:w'}
\begin{split}
\sup_{t\in[\ell,M]}
\left|w'_s(t;\beta_0)-\Phi'(t;\beta_0)\right|
&=
\sup_{t\in[\ell,M]}
\left|\int_{-1}^1 k(y)\Phi'(t-by;\beta_0)\,\mathrm{d}y-\Phi'(t;\beta_0)\right|\\
&\leq
\sup_{t\in[\ell,M]}
\int_{-1}^1 |k(y)|
\left|
\Phi'(t-by;\beta_0)-\Phi'(t;\beta_0)
\right| \,\mathrm{d}y\to 0,
\end{split}
\end{equation}
by uniform continuity of $\Phi'$.
Furthermore, with~\eqref{eqn:w'_n}, we obtain
\[
\sup_{t\in[\ell,M]}
\left|w'_n(t;\hat{\beta}_n)-w'_s(t;\beta_0)\right|
\leq
\frac{1}{b}\sup_{x\in\R}|\Phi_n(x;\hat{\beta}_n)-\Phi(x;\beta_0)|
\int_{-1}^1|k'(y)|\,\mathrm{d}y\to 0,
\]
because
\begin{equation}
\label{eq:Phihat-Phi0}
\begin{split}
\sup_{x\in\R}|\Phi_n(x;\hat{\beta}_n)-\Phi(x;\beta_0)|
&\leq
\sup_{x\in\R}|\Phi_n(x;\hat{\beta}_n)-\Phi_n(x;\beta_0)|
+
\sup_{x\in\R}|\Phi_n(x;\beta_0)-\Phi(x;\beta_0)|\\
&\leq
\sup_{x\in\R}
\left|
\frac{\partial\Phi_n(x;\beta_n^*)}{\partial\beta}
\right|(\hat\beta_n-\beta_0)
+
O_p(n^{-1/2})
=
O_p(n^{-1/2}),
\end{split}
\end{equation}
due to Lemmas 3 and 4 in~\cite{LopuhaaNane2013}.
Together with~\eqref{eqn:w'} this proves the last result.
\end{proof}

\begin{lemma}
\label{lem:monotone}
Suppose that (A1)-(A2) hold.
Let $H^{uc}(t)$ and $\Phi(t;\beta_0)$ be defined in~\eqref{eq:def Huc} and~\eqref{eq:def Phi}, and let $h(t)=\mathrm{d}H^{uc}(t)/\mathrm{d}t$,
satisfying~\eqref{eqn:lambda0}.
Suppose that $h$ and $t\mapsto\Phi(t;\beta_0)$ are continuously differentiable,
and that~$\lambda'_0$ is uniformly bounded from below by a strictly positive constant.
Let $k$ satisfy~\eqref{def:kernel} and let~$\hat{\lambda}_n^{\mathrm{naive}}$ be defined in~\eqref{def:naive est MSLE}.
If $b\to0$ and $1/b=O(n^{\alpha})$, for some $\alpha\in(0,1/4)$,
then for each $0<\ell<M<\tau_H$,  it holds
\[
\p\big(\hat{\lambda}_n^{\mathrm{naive}}\text{ is increasing on } [\ell,M]\big)\to 1.
\]
\end{lemma}
\begin{proof}
Note that $w_n(x,\hat{\beta}_n)=0$ if and only if $T_i\leq x-b$,
for all $i=1,\ldots,n$, which happens with probability $H(x-b)^n\leq H(M)^n\to 0$.
This means that with probability tending to one, $w_n(x,\hat{\beta}_n)>0$ for all $x\in[\ell,M]$.
Thus with probability tending to one,
$\hat{\lambda}_n^{\mathrm{naive}}$ is well defined on $[\ell,M]$ and
\begin{equation}
\label{eqn:der.naive}
\frac{\mathrm{d}}{\mathrm{d}x}\hat{\lambda}_n^{\mathrm{naive}}(x)=\frac{v'_n(x)\,w_n(x;\hat{\beta}_n)-v_n(x)\,w'_n(x;\hat{\beta}_n)}{w_n(x;\hat{\beta}_n)^2}.
\end{equation}
In order to prove that $\hat{\lambda}_n^{\mathrm{naive}} $ is increasing on $[\ell,M]$ with probability tending to one, it suffices to show that
\begin{equation}
\label{eq:prob inf}
\p\left(
\inf_{x\in[\ell,M]}
\left\{
v'_n(x)w_n(x;\hat{\beta}_n)-v_n(x)w'_n(x;\hat{\beta}_n)
\right\}\leq 0
\right)\to 0.
\end{equation}
We can write
\[
\begin{split}
&
v'_n(x)w_n(x;\hat{\beta}_n)-v_n(x)w'_n(x;\hat{\beta}_n)\\
&=
w_n(x;\hat{\beta}_n)
\left(v'_n(x)-h'(x)\right)+v_n(x)\left(\Phi'(x;\beta_0)-w'_n(x;\hat{\beta}_n)\right)\\
&\qquad+
h'(x)\left(w_n(x;\hat{\beta}_n)-\Phi(x;\beta_0)\right)+\Phi'(x;\beta_0)\left(h(x)-v_n(x)\right)\\
&\qquad\qquad+
h'(x)\Phi(x;\beta_0)-\Phi'(x;\beta_0)h(x),
\end{split}
\]
where the right hand side can be bounded from below by
\[
\begin{split}
&
-\sup_{x\in[\ell,M]}|v'_n(x)-h'(x)|\sup_{x\in[\ell,M]}|w_n(x;\hat{\beta}_n)|\\
&\qquad-
\sup_{x\in[\ell,M]}|\Phi'(x;\beta_0)-w'_n(x;\hat{\beta}_n)|\sup_{x\in[\ell,M]}|v_n(x)|\\
&\qquad-
\sup_{x\in[\ell,M]}|w_n(x;\hat{\beta}_n)-\Phi(x;\beta_0)|\sup_{x\in[\ell,M]}|h'(x)|\\
&\qquad-
\sup_{x\in[\ell,M]}|h(x)-v_n(x)|\sup_{x\in[\ell,M]}|\Phi'(x;\beta_0)|
+
h'(x)\Phi(x;\beta_0)-\Phi'(x;\beta_0)h(x).
\end{split}
\]
From the proof of Lemma~\ref{le:2} we have that $\sup_{x\in[\ell,M]}|v_n(x)|$ and $\sup_{x\in[\ell,M]}w_n(x;\hat{\beta}_n)$ are~$O_p(1)$,
so that from Lemmas~\ref{le:1a} and~\ref{le:1} (with $m=1$), it follows that the first four terms on the right hand side tend to zero in probability.
Therefore, the probability in~\eqref{eq:prob inf} is bounded by
\[
\p
\left(
X_n\geq \inf_{x\in[\ell,M]}
\left(
h'(x)\Phi(x;\beta_0)-\Phi'(x;\beta_0)h(x)
\right)
\right),
\]
where $X_n=o_p(1)$.
This probability tends to zero, because with~\eqref{eqn:lambda0}, we have
\[
\begin{split}
\inf_{x\in[\ell,M]}
\left(
h'(x)\Phi(x;\beta_0)-\Phi'(x;\beta_0)h(x)
\right)
&=
\inf_{x\in[\ell,M]}\lambda'_0(x)\Phi(x;\beta_0)^2\\
&\geq
\Phi(M;\beta_0)^2
\inf_{x\in[0,\tau_H]}\lambda'_0(x)
>0.
\end{split}
\]
\end{proof}

\begin{lemma}
\label{le:2}
Let $\tilde W_n$, $\tilde V_n$, and $W_0$ be defined by~\eqref{def:tilde Wn tilde V_n} and~\eqref{def:W0},
and
let $H^{uc}$ be defined in~\eqref{eq:def Huc}.
If $b\to0$ and $1/b=O(n^{-1/2})$, then, under the assumptions of Lemma~\ref{le:1} with $m=1$, it holds
\begin{equation}
\label{eqn:lemma2}
\sup_{t\in[0,\tau_H]}
\left|\tilde V_n(t)-H^{uc}(t)\right|\xrightarrow{\p}0,
\qquad
\sup_{t\in[0,\tau_H]}
\left|\tilde W_n(t)-W_0(t)
\right|\xrightarrow{\p}0.
\end{equation}
\end{lemma}
\begin{proof}
To prove the first result in~\eqref{eqn:lemma2}, we take $0<\epsilon<\tau_H$ arbitrarily and write
\begin{equation}
\label{eq:bound Huc minus Vn}
\begin{split}
&\sup_{t\in[0,\tau_H]}|\tilde V_n(t)-H^{uc}(t)|
\leq
\int_0^{\tau_H} |v_n(u)-h(u)|\,\mathrm{d}u\\
&=\int_0^\epsilon |v_n(u)-h(u)|\,\mathrm{d}u+\int_\epsilon^{\tau_H-\epsilon} |v_n(u)-h(u)|\,\mathrm{d}u+\int_{\tau_H-\epsilon}^{\tau_H} |v_n(u)-h(u)|\,\mathrm{d}u\\
&\leq
2\epsilon\sup_{u\in[0,\tau_H]}|v_n(u)|
+
2\epsilon\sup_{u\in[0,\tau_H]}|h(u)|
+
(\tau_H-2\epsilon)\sup_{u\in[\epsilon,\tau_H-\epsilon]}|v_n(u)-h(u)|.
\end{split}
\end{equation}
Since $h$ is bounded and the last term tends to zero in probability, according to Lemma~\ref{le:1} with $m=1$,
it suffices to prove that $\sup_{u\in[0,\tau_H]}|v_n(u)|=O_p(1)$.
By definition and the triangular inequality we have
\[
\begin{split}
|v_n(t)|
&=
\left|
\int
k_b\left(t-u\right)\,\mathrm{d}H^{uc}_n(u)\right|\\
&\leq
b^{-1}\left|
H^{uc}_n((t+b)\wedge \tau_H)-H^{uc}_n((t-b)\vee 0)
\right|
\sup_{y\in[-1,1]}|k(y)|\\
&\leq
b^{-1}
\Bigg\{
\left|
H^{uc}_n((t+b)\wedge \tau_H)-H^{uc}((t+b)\wedge \tau_H)
-
H^{uc}_n((t-b)\vee 0)+H^{uc}((t-b)\vee 0)\right|\\
&\qquad\qquad\qquad\qquad+
H^{uc}((t+b)\wedge \tau_H)-H^{uc}((t-b)\vee 0)
\Bigg\}
\sup_{y\in[-1,1]}|k(y)|\\
&\leq
2
\left\{
b^{-1}
\sup_{y\in[0,\tau_H]}|H^{uc}_n(y)-H^{uc}(y)|
+
2
\sup_{u\in[0,\tau_H]}|h(u)|
\right\}
\sup_{y\in[-1,1]}|k(y)|.
\end{split}
\]
Using~\eqref{eqn:H_uc}, it follows that the right hand side of the previous inequality is bounded in probability.
For the second result in~\eqref{eqn:lemma2}, similar to~\eqref{eq:bound Huc minus Vn} we have
\[
\begin{split}
\sup_{t\in[0,\tau_H]}
\left|\tilde W_n(t)-W_0(t)
\right|
&\leq
\int_0^{\tau_H}
\left|w_n(u;\hat\beta_n)-\Phi(u;\beta_0)\right|
\,\mathrm{d}u\\
&\leq
2\epsilon\sup_{u\in[0,\tau_H]}|w_n(u;\hat\beta_n)|
+
2\epsilon\sup_{u\in[0,\tau_H]}|\Phi(u;\beta_0)|\\
&\qquad+
(\tau_H-2\epsilon)\sup_{u\in[\epsilon,\tau_H-\epsilon]}\left|w_n(u;\hat\beta_n)-\Phi(u;\beta_0)\right|.
\end{split}
\]
By using Lemma~\ref{le:1} with $m=1$ and the fact that $\Phi(u;\beta_0)$ is bounded, it remains to
handle the first term on right hand side.
Since
\begin{equation}
\label{eq:bound int kb}
\left|
\int_{t}^{\infty}k_b(s-u)\,\mathrm{d}s
\right|
=
\left|
\int_{(t-u)/b}^{\infty}k(y)\,\mathrm{d}y
\right|
\leq
2\sup_{y\in[-1,1]}|k(y)|,
\end{equation}
and $k_b(t-u)=0$, for $u<t-b$,
we have
\[
\begin{split}
|w_n(t;\hat\beta_n)|
&=
\left|
\int
\mathrm{e}^{\hat{\beta}'_nz}
\int_{t}^{\infty}k_b(s-u)\,\mathrm{d}s
\,\mathrm{d}\p_n(u,\delta,z)
\right|
\\
&\leq
2\sup_{y\in[-1,1]}|k(y)|
\int_{t-b}^{\infty}\mathrm{e}^{\hat{\beta}'_nz}\,\mathrm{d}\p_n(u,\delta,z)
=
2\Phi_n(t-b;\hat{\beta}_n)
\sup_{y\in[-1,1]}|k(y)|,
\end{split}
\]
whereas Lemma~3 in~\cite{LopuhaaNane2013} gives that $\sup_{x\in\R}\Phi_n(x;\hat{\beta}_n)=O_p(1)$.
This establishes the second result in~\eqref{eqn:lemma2}.
\end{proof}

\begin{proof}[Proof of Corollary~\ref{cor:MSLE=naive}]
According to Lemma~\ref{lem:cons-naive}, Lemma~\ref{lem:monotone}, and Lemma~\ref{le:2} together with~\eqref{eqn:lambda0},
conditions~\eqref{cond b}-\eqref{cond d} of Lemma~\ref{lem:coincide} are satisfied, with
$X_n=\tilde W_n$, $Y_n=\tilde V_n$, and~$\hat\tau=\sup\{t\geq 0:w_n(t;\hat\beta_n)>0\}$,
and condition~\eqref{cond a} of Lemma~\ref{lem:coincide} is trivially fulfilled with $X_n=\tilde W_n$.
Hence, the corollary follows from Lemma~\ref{lem:coincide}.
\end{proof}

\begin{lemma}
\label{lemma:distr}
Suppose that (A1)-(A2) hold.
Fix $x\in(0,\tau_H)$.
Let $H^{uc}(t)$ and $\Phi(t:\beta_0)$ be defined in~\eqref{eq:def Huc} and~\eqref{eq:def Phi}, and let $h(t)=\mathrm{d}H^{uc}(t)/\mathrm{d}t$,
satisfying~\eqref{eqn:lambda0}.
Suppose that $h$ and $t\mapsto\Phi(t;\beta_0)$ are $m\geq 2$ times continuously differentiable
and that~$\lambda'_0$ is uniformly bounded from below by a strictly positive constant.
Let $k$ be $m$-orthogonal satisfying~\eqref{def:kernel}.
Let~$v_n$ and~$w_n$ be defined in~\eqref{eqn:v_n w_n} and suppose that $n^{1/(2m+1)}b\to c>0$.
Then
\[
n^{m/(2m+1)}
\left(
\begin{bmatrix}
w_n(x;\hat\beta_n)\\
v_n(x)
\end{bmatrix}
-
\begin{bmatrix}
\Phi(x;\beta_0)\\
h(x)
\end{bmatrix}
\right)
\to
N\left(\begin{bmatrix}
             \mu_1 \\
             \mu_2 \\
           \end{bmatrix},\begin{bmatrix}
0 & 0\\
0 & \sigma^2
\end{bmatrix}\right)
\]
where
\[
\begin{bmatrix}
             \mu_1 \\
             \mu_2 \\
           \end{bmatrix}
=
\frac{(-c)^m}{m!}\int_{-1}^1 k(y)y^m\,\mathrm{d}y
\begin{bmatrix}
\Phi^{(m)}(x;\beta_0)\\
h^{(m)}(x)
\end{bmatrix},
\qquad
\sigma^2=\frac{h(x)}c
\int_{-1}^1 k^2(y)\,\mathrm{d}y.
\]
\end{lemma}
\begin{proof}
First we show that $n^{m/(2m+1)}(w_n(x;\hat{\beta}_n)-w_n(x;\beta_0))\to 0$ in probability,
which enables us to replace $w_n(x;\hat{\beta}_n)$ with $w_n(x;\beta_0)$ in the statement.
From~\eqref{eqn:v_n w_n}, together with~\eqref{eq:bound int kb}, we find
\[
\begin{split}
\left|w_n(x;\hat{\beta}_n)-w_n(x;\beta_0)\right|
&\leq
\frac1n
\sum_{i=1}^n \left|\mathrm{e}^{\hat{\beta}'_nZ_i}-\mathrm{e}^{\beta'_0Z_i}\right|
\left|
\int_x^{\infty}k_b(u-T_i)\,\mathrm{d}u
\right|\\
&\leq
2\sup_{y\in[-1,1]}|k(y)|
\frac1n
\sum_{i=1}^n
|Z_i|\mathrm{e}^{\tilde{\beta}'_{n,i}Z_i}\left|\hat{\beta}_n-\beta_0\right|,
\end{split}
\]
for some $|\tilde{\beta}_{n,i}-\beta_0|\leq|\hat{\beta}_n-\beta_0|=O_p(n^{-1/2})$.
Furthermore, for all $M>0$,
\[
\p\left(
\frac{1}{n}\sum_{i=1}^n
|Z_i|\mathrm{e}^{\tilde{\beta}'_{n,i}Z_i}\geq M
\right)
\leq
\frac{1}{nM}\sum_{i=1}^n
\E
\left[|Z_i|\mathrm{e}^{\tilde{\beta}'_{n,i}Z_i} \right]
\leq
\frac{1}{M}\sup_{|\beta-\beta_0|\leq \epsilon}\E\left[|Z|\mathrm{e}^{\beta'Z} \right],
\]
where $\sup_{|\beta-\beta_0|\leq \epsilon}\E[|Z|\mathrm{e}^{\beta'Z}]<\infty$ according to assumption~(A2).
It follows that
\[
n^{m/(2m+1)}
\left(
w_n(x;\hat{\beta}_n)-w_n(x;\beta_0)
\right)
=
O_p(n^{-1/(4m+2)}).
\]

Now, define
\[
Y_{ni}=
\begin{bmatrix}
Y_{ni,1}\\
Y_{ni,2}
\end{bmatrix}
=
n^{-(m+1)/(2m+1)}
\begin{bmatrix}
\mathrm{e}^{\beta'_0Z_i}\int_x^{\infty}k_b(s-T_i)\,\mathrm{d}s
\\
k_b(x-T_i)\Delta_i
\end{bmatrix}.
\]
By a Taylor expansion, using that
$h$ is $m$ times continuously differentiable and that $k$ is $m$-orthogonal, as in~\eqref{eq:taylor h} we obtain
\begin{equation}
\label{eq:expec Yni2}
\begin{split}
\E\left[Y_{ni,2}\right]
&=
n^{-(m+1)/(2m+1)}
\int_{-1}^1 k(y)h(x-by)\,\mathrm{d}y\\
&=
n^{-(m+1)/(2m+1)}
\left(h(x)+\frac{(-b)^m}{m!}h^{(m)}(x)\int_{-1}^1k(y)y^m\,\mathrm{d}y + o(b^m)\right).
\end{split}
\end{equation}
Similarly, with Fubini we get
\begin{equation}
\label{eq:expec Yni1}
\begin{split}
\E\left[Y_{ni,1}\right]
&=
n^{-(m+1)/(2m+1)}
\int \mathrm{e}^{\beta'_0z}
\int_x^{\infty}k_b(s-u)\,\mathrm{d}s\,\mathrm{d}\p(u,\delta,z)\\
&=
n^{-(m+1)/(2m+1)}
\int_{-1}^1
\left(\int
\mathrm{e}^{\beta'_0z}\1_{\{u\geq x-by\}}\,\mathrm{d}\p(u,\delta,z)
\right)
k(y)\,\mathrm{d}y\\
&=
n^{-(m+1)/(2m+1)}
\int_{-1}^1 k(y)\Phi(x-by;\beta_0)\,\mathrm{d}y\\
&=
n^{-(m+1)/(2m+1)}
\left(
\Phi(x;\beta_0)+\frac{(-b)^m}{m!}\Phi^{(m)}(x;\beta_0)\int_{-1}^1 k(y)y^m\,\mathrm{d}y + o(b^m)
\right).
\end{split}
\end{equation}
Hence, we have
\[
\E\left[Y_{ni}\right]
=
n^{-(m+1)/(2m+1)}
\begin{bmatrix}
\Phi(x;\beta_0)\\
h(x)
\end{bmatrix}
+
n^{-1}
\begin{bmatrix}
\mu_1 \\
\mu_2 \\
\end{bmatrix}
+
o(n^{-1}),
\]
and we can write
\[
n^{-(m+1)/(2m+1)}
\left(
\begin{bmatrix}
w_n(x;\hat\beta_n)\\
v_n(x)
\end{bmatrix}
-
\begin{bmatrix}
\Phi(x;\beta_0)\\
h(x)
\end{bmatrix}
\right)
=
\begin{bmatrix}
\mu_1 \\
\mu_2 \\
\end{bmatrix}
+
\sum_{i=1}^n
\big(
Y_{ni}-\E\left[Y_{ni}\right]
\big)
+
o(1).
\]
It remains to show that $\sum_{i=1}^n \big(Y_{ni}-\E\left[Y_{ni}\right]\big)$ converges in distribution to a bivariate normal distribution with mean zero.
From~\eqref{eq:expec Yni1} we have,
\begin{equation}
\label{eq:EYn1^2}
\begin{split}
&
\mathrm{Var}(Y_{ni,1})
=
\E\left[Y_{ni,1}^2\right]+O(n^{-2(m+1)/(2m+1)})\\
&=
n^{-2(m+1)/(2m+1)}
\int \mathrm{e}^{2\beta'_0z}
\left(\int_x^{\infty}k_b(s-u)\,\mathrm{d}s\right)^2\,\mathrm{d}\p(u,\delta,z)+O(n^{-2(m+1)/(2m+1)})\\
&=
O(n^{-2(m+1)/(2m+1)}),
\end{split}
\end{equation}
using that, with~\eqref{eq:bound int kb},
\[
\begin{split}
\int
\mathrm{e}^{2\beta'_0z}\left(\int_x^{\infty}k_b(s-u)\,\mathrm{d}s\right)^2\,\mathrm{d}\p(u,\delta,z)
&\leq
\left(
2\sup_{y\in[-1,1]}|k(y)|
\right)^2
\int \mathrm{e}^{2\beta'_0z}\,\mathrm{d}\p(u,\delta,z)\\
&=
\left(
2\sup_{y\in[-1,1]}|k(y)|
\right)^2
\Phi(0;2\beta_0)<\infty.
\end{split}
\]
Moreover,
\[
\begin{split}
&
\mathrm{Cov}(Y_{ni,1},Y_{ni,2})
=
\E\left[Y_{ni,1}Y_{ni,2}\right]+O(n^{-2(m+1)/(2m+1)})\\
&=
n^{-2(m+1)/(2m+1)}
\int \delta\mathrm{e}^{\beta'_0z}
\left(\int_x^{\infty}k_b(s-u)\,\mathrm{d}s\right)
k_b(x-u)\,\mathrm{d}\p(u,\delta,z)+O_p(n^{-2(m+1)/(2m+1)})\\
&=
o(n^{-1})+O(n^{-2(m+1)/(2m+1)}),
\end{split}
\]
because, with~\eqref{eq:bound int kb},
\[
\begin{split}
&
\left|
b\int \delta\mathrm{e}^{\beta'_0z}
\left(\int_x^{\infty}k_b(s-u)\,\mathrm{d}s\right)
k_b(x-u)\,\mathrm{d}\p(u,\delta,z)
\right|
\\
&\leq
2\sup_{y\in[-1,1]}|k(y)|
\int \1_{\{x-b\leq u\leq x+b\}}
\mathrm{e}^{\beta'_0z}
\left|
k\left(\frac{x-u}{b}\right)
\right|\,\mathrm{d}\p(u,\delta,z)\\
&\leq
2\left(\sup_{y\in[-1,1]}|k(y)|\right)^2
\big( \Phi(x-b;\beta_0)-\Phi(x+b;\beta_0)\big)
\to 0.
\end{split}
\]
Once again, by a Taylor expansion, from~\eqref{eq:expec Yni2}, we obtain
\begin{equation}
\label{eq:EYn2^2}
\begin{split}
\text{Var}(Y_{ni,2})
&=
\E
\left[Y_{ni,2}^2\right]+O(n^{-2(m+1)/(2m+1)})\\
&=
n^{-2(m+1)/(2m+1)}b^{-1}
\int_{-1}^1 k^2(y)h(x-by)\,\mathrm{d}y+O(n^{-2(m+1)/(2m+1)})\\
&=
n^{-1}\sigma^2+o(n^{-1}).
\end{split}
\end{equation}
It follows that
\[
\sum_{i=1}^n
\mathrm{Cov}( Y_{ni})
=
\begin{bmatrix}
 0 & 0\\
 0 & \sigma^2
\end{bmatrix}+o(1).
\]
Furthermore, since
\[
|Y_{ni}|^2
=
n^{-2(m+1)/(2m+1)}\left(
\mathrm{e}^{2\beta'_0Z_i}
\left(\int_x^{\infty}k_b(s-T_i)\,\mathrm{d}s\right)^2
+
k^2_b(x-T_i)\Delta_i
\right),
\]
with~\eqref{eq:bound int kb}, we obtain
\[
\begin{split}
\sum_{i=1}^n
\E\left[|Y_{ni}|^2\1_{\{|Y_{ni}|>\epsilon \}} \right]
&\leq
\left(2\sup_{y\in[-1,1]}|k(y)|\right)^2
n^{-1/(2m+1)}\E\left[\mathrm{e}^{2\beta'_0Z}\right]\\
&\qquad+
n^{-2(m+1)/(2m+1)}b^{-2}\sup_{y\in[-1,1]}|k(y)|
\sum_{i=1}^n\p\left(|Y_{ni}|>\epsilon\right),
\end{split}
\]
where the right hand side tends to zero, because $\E[\mathrm{e}^{2\beta'_0Z}]=\Phi(0;2\beta_0)<\infty$ and,  with~\eqref{eq:EYn1^2} and~\eqref{eq:EYn2^2},
we have
\[
\sum_{i=1}^n
\p\left(|Y_{ni}|>\epsilon\right)
\leq
\epsilon^{-2}
\sum_{i=1}^n
\E|Y_{ni}|^2
=
O(1).
\]
By the multivariate Lindeberg-Feller central limit theorem, we get
\[
\sum_{i=1}^n \big(Y_{ni}-\E\left[Y_{ni}\right]\big)
\stackrel{d}{\to}
N\left(\begin{bmatrix}
             \mu_1 \\
             \mu_2 \\
           \end{bmatrix},\begin{bmatrix}
0 & 0\\
0 & \sigma^2
\end{bmatrix}\right),
\]
which finishes the proof.
\end{proof}

\begin{proof}[Proof of Theorem~\ref{theo:distrMS}]
By definition of $\hat{\lambda}_n^{\mathrm{naive}}(x)$ in~\eqref{def:naive est MSLE} together with~\eqref{eqn:lambda0},
we can write
\[
\hat{\lambda}_n^{\mathrm{naive}}(x)-\lambda_0(x)
=
\phi\Big(w_n(x;\hat\beta_n),v_n(x)\Big)
-
\phi\Big(\Phi(x;\beta_0),\lambda_0(x)\Phi(x;\beta_0)\Big)
\]
with $\varphi(w,v)=v/w$.
The asymptotic distribution of $\hat{\lambda}_n^{\mathrm{naive}}(x)$ then follows from an application of the delta method to the result in Lemma~\ref{lemma:distr}.
Then, by Corollary~\ref{cor:MSLE=naive}, this also gives the asymptotic distribution of $\hat{\lambda}^{MS}_n(x)$.
\end{proof}

\begin{proof}[Proof of Theorem~\ref{theo:MSLE asymptotic equivalence}]
First note that by means of~\eqref{eqn:lambda0}, it follows from the assumptions of the theorem that~$h(t)=\mathrm{d}H^{uc}(t)/\mathrm{d}t$ is $m\geq 2$ times
continuously differentiable.
We write
\[
\begin{split}
&
n^{m/(2m+1)}
\left(
\hat{\lambda}_n^{MS}(x)-\tilde{\lambda}_n^{SM}(x)
\right)\\
&=
n^{m/(2m+1)}
\left(
\hat{\lambda}_n^{\mathrm{naive}}(x)-\tilde{\lambda}_n^{SM}(x)
\right)
+
n^{m/(2m+1)}
\left(
\hat{\lambda}_n^{MS}(x)-\hat{\lambda}_n^{\mathrm{naive}}(x)
\right).
\end{split}
\]
By Corollary~\ref{cor:MSLE=naive}, the second term on the right hand side converges to zero in probability.
Furthermore, as can be seen from the proof of Theorem~3.5 in~\cite{LopuhaaMustaSI2016},
\[
n^{m/(2m+1)}
\left(
\tilde{\lambda}^{SM}_n(x)-\lambda_0(x)
\right)
=
\widetilde{\mu}
+
n^{m/(2m+1)}
\int
\frac{\delta k_b(x-u)}{\Phi(u;\beta_0)}\,\mathrm{d}(\p_n-\p)(u,\delta,z)+o_p(1),
\]
with $\widetilde{\mu}$ from~\eqref{def:mutilde}.
From the proof of Lemma~\ref{lemma:distr}, we have
\[
\hat{\lambda}_n^{\mathrm{naive}}(x)-\lambda_0(x)
=
\phi\Big(w_n(x;\hat\beta_n),v_n(x)\Big)
-
\phi\Big(\Phi(x;\beta_0),\lambda_0(x)\Phi(x;\beta_0)\Big),
\]
where $\phi(w,v)=v/w$ and
\[
n^{m/(2m+1)}
\left(
\begin{bmatrix}
w_n(x;\hat\beta_n)\\
v_n(x)
\end{bmatrix}
-
\begin{bmatrix}
\Phi(x;\beta_0)\\
h(x)
\end{bmatrix}
\right)
=
\begin{bmatrix}
\mu_1\\
\mu_2
\end{bmatrix}
+
\begin{bmatrix}
Z_{n1}\\
Z_{n2}
\end{bmatrix}
+
o(1),
\]
with $Z_{n1}=o_P(1)$ and
\[
\begin{bmatrix}
\mu_1\\
\mu_2
\end{bmatrix}
=
\frac{(-c)^m}{m!}\int_{-1}^1 k(y)y^m\,\mathrm{d}y
\begin{bmatrix}
\Phi^{(m)}(x;\beta_0)\\
h^{(m)}(x)
\end{bmatrix}.
\]
Then with a Taylor expansion it follows that
\[
\begin{split}
n^{m/(2m+1)}
\left(
\hat{\lambda}_n^{\mathrm{naive}}(x)-\lambda_0(x)
\right)
&=
\begin{bmatrix}
  -\dfrac{\lambda_0(x)}{\Phi(x;\beta_0)} & \dfrac1{\Phi(x;\beta_0)}
\end{bmatrix}
\left(
\begin{bmatrix}
\mu_1\\
\mu_2
\end{bmatrix}
+
\begin{bmatrix}
Z_{n1}\\
Z_{n2}
\end{bmatrix}
\right)
+
o_p(1)\\
&=
\mu
+
\frac{Z_{n2}}{\Phi(x;\beta_0)}
+
o_p(1),
\end{split}
\]
where $\mu$ is from Theorem~\ref{theo:distrMS}.
Moreover, from the proof of Lemma~\ref{lemma:distr} it can be seen that
\[
\begin{split}
\frac{Z_{n2}}{\Phi(x;\beta_0)}
&=
\frac1{\Phi(x;\beta_0)}
n^{m/(2m+1)}
\int \delta k_b(x-u)\,(\p_n-\p)(u,\delta,z)+o_P(1)\\
&=
n^{m/(2m+1)}
\int \frac{\delta k_b(x-u)}{\Phi(u;\beta_0)}\,(\p_n-\p)(u,\delta,z)\\
&\quad
+
n^{m/(2m+1)}
\int \delta k_b(x-u)
\left(
\frac{1}{\Phi(x;\beta_0)}
-
\frac{1}{\Phi(u;\beta_0)}
\right)
\,(\p_n-\p)(u,\delta,z)+o_P(1)\\
&=
n^{m/(2m+1)}
\int \frac{\delta k_b(x-u)}{\Phi(u;\beta_0)}\,(\p_n-\p)(u,\delta,z)
+o_P(1),
\end{split}
\]
because
\[
n^{m/(2m+1)}
\int \delta k_b(x-u)
\left(
\frac{1}{\Phi(x;\beta_0)}
-
\frac{1}{\Phi(u;\beta_0)}
\right)
\,(\p_n-\p)(u,\delta,z)
=
\sum_{i=1}^n
\left(
X_{ni}-\E\left[X_{ni}\right]
\right)
\]
with
\[
X_{ni}
=
n^{-(m+1)/(2m+1)}
\Delta_i k_b(x-T_i)
\left(
\frac{1}{\Phi(x;\beta_0)}
-
\frac{1}{\Phi(T_i;\beta_0)}
\right),
\]
where similar to the proof of Lemma~\ref{lemma:distr},
\[
\begin{split}
\E\left[X_{ni}^2\right]
&=
n^{-2(m+1)/(2m+1)}
\int k_b^2(x-u)
\left(
\frac{1}{\Phi(x;\beta_0)}
-
\frac{1}{\Phi(u;\beta_0)}
\right)^2 h(u)\,\mathrm{d}u\\
&=
n^{-2(m+1)/(2m+1)}b^{-1}
\int k^2(y)
\left(
\frac{1}{\Phi(x;\beta_0)}
-
\frac{1}{\Phi(x-by;\beta_0)}
\right)^2 h(x-by)\,\mathrm{d}y\\
&=
o(n^{-1}).
\end{split}
\]
We conclude that
\[
\begin{split}
n^{m/(2m+1)}
\left(
\hat{\lambda}_n^{\mathrm{naive}}(x)-\lambda_0(x)
\right)
&=
\mu
+
n^{m/(2m+1)}
\int \frac{\delta k_b(x-u)}{\Phi(u;\beta_0)}\,(\p_n-\p)(u,\delta,z)
+o_P(1)\\
&=
\mu-\widetilde{\mu}
+
n^{m/(2m+1)}
\left(
\tilde{\lambda}^{SM}_n(x)-\lambda_0(x)
\right)
+
o_p(1)
\end{split}
\]
which proves the first statement in the theorem.
The second statement is immediate using the asymptotic equivalence in~\eqref{eq:asymp norm SM SG}.
\end{proof}

\subsection{Proofs for Section~\ref{sec:GS}}
\label{subsec:proofs GS}
\begin{lemma}
\label{lem:monotone2}
Suppose that (A1)-(A2) hold.
Let $\lambda_0$ be continuously differentiable, with $\lambda'_0$ uniformly bounded from below by a strictly positive constant,
and let $k$ satisfy~\eqref{def:kernel}.
If $b\to0$ and $1/b=O(n^{\alpha})$, for some $\alpha\in(0,1/4)$, then for each $0<\ell<M<\tau^*$,  it holds
\[
\p
\left(
\tilde{\lambda}_n^{\mathrm{naive}}\text{ is increasing on } [\ell,M]
\right)\to 1.
\]
\end{lemma}
\begin{proof}
From~\eqref{def:naive est GS}, it follows with integration by parts that
\begin{equation}
\label{eq:decomp naive GS}
\begin{split}
\tilde{\lambda}_n^{\mathrm{naive}}(x)
&=
\int k_b'(x-u)\Lambda_0(u)\,\mathrm{d}u
+
\int k_b'(x-u)
\left(
\Lambda_n(u)-\Lambda_0(u)
\right)\,\mathrm{d}u\\
&=
\lambda_0(x)+\int k_b(x-u)
\big\{
\lambda_0(u)-\lambda_0(x)
\big\}\,\mathrm{d}u
+
\int k_b(x-u)
\,\mathrm{d}\left(\Lambda_n-\Lambda_0\right)(u),\\
\end{split}
\end{equation}
so that
\begin{equation}
\label{eqn:der.naive2}
\begin{split}
\frac{\mathrm{d}}{\mathrm{d}x}\tilde{\lambda}_n^{\mathrm{naive}}(x)
=
\lambda'_0(x)
&+
\int_{-1}^1 k(y)
\big\{
\lambda'_0(x-by)-\lambda'_0(x)
\big\}\,\mathrm{d}y\\
&+
\frac{1}{b^2}\int k'\left(\frac{x-u}{b}\right)\,\mathrm{d}\left(\Lambda_n-\Lambda_0\right)(u).
\end{split}
\end{equation}
By assumption, the first term on the right hand side of~\eqref{eqn:der.naive2} is
bounded from below by a strictly positive constant and the second term converges to zero because of the continuity of~$\lambda'_0$.
Moreover, let $0<M<M'<\tau_H$, so that for $n$ sufficiently large $M+b<M'$.
Then, the second term on the right hand side of~\eqref{eqn:der.naive2} is bounded from above in absolute value by
\[
\frac{1}{b^2}\sup_{x\in[0,M']}\left|\Lambda_n(x)-\Lambda_0(x)\right|
\sup_{y\in[-1,1]}|k''(y)|
=
O_p(n^{2\alpha-1/2})=o_p(1),
\]
according to~\eqref{eqn:Breslow} and the fact that $\alpha<1/4$.
We conclude that $\tilde{\lambda}_n^{\mathrm{naive}}$ is increasing on $[\ell,M]$ with probability tending to one.
\end{proof}

\begin{proof}[Proof of Corollary~\ref{cor:ISBE=naive}]
We apply Lemma~\ref{lem:coincide}.
Condition~\eqref{cond a} is trivial with $X_n(t)=t$.
Furthermore, for every fixed $t\in(0,\tau^*)$,
we have for sufficiently large $n$, that $t\in(b,\tau^*-b)$ and
\begin{equation}
\label{eq:Lemma53 GS}
\begin{split}
&
\tilde{\lambda}_n^{\mathrm{naive}}(t)-\lambda_0(t)\\
&=
\int k_b(t-u)\lambda_0(u)\,\mathrm{d}u-\lambda_0(t)
+
\int k_b(t-u)\,\mathrm{d}\left(\Lambda_n(u)-\Lambda_0(u)\right)\\
&=
\int_{-1}^1 k(y)
\left\{
\lambda_0(t-by)-\lambda_t(x)
\right\}\,\mathrm{d}y
+
b^{-1}
\int_{-1}^1
\left(\Lambda_n(t-by)-\Lambda_0(t-by)\right)
k'(y)\,\mathrm{d}y\\
&=
o_p(1),
\end{split}
\end{equation}
by continuity of $\lambda_0$ and~\eqref{eqn:Breslow}, which proves condition~\eqref{cond b} of Lemma~\ref{lem:coincide}.
Condition~\eqref{cond c} follows from Lemma~\ref{lem:monotone2}.
Finally, for $t\in[0,\tau^*]$,
\[
\begin{split}
&
\left|\hat\Lambda_n^s(t)-\Lambda_0(t)\right|\\
&=
\left|\int_{(t-b)\vee 0}^{(t+b)\wedge \tau_H} k_b(t-u)
\left(\Lambda_n(u)-\Lambda_0(u)\right)\,\mathrm{d}u
+
\int_{(t-b)\vee 0}^{(t+b)\wedge \tau_H}k_b(t-u)\Lambda_0(u)\,\mathrm{d}u-\Lambda_0(t)
\right|\\
&\leq \left|
\int_{-1}^{t/b\wedge 1} k(y)
\left(\Lambda_n(t-by)-\Lambda_0(t-by)\right)\,\mathrm{d}y\right|
+
\left|
\int_{-1}^{t/b\wedge 1} k(y)\Lambda_0(t-by)\,\mathrm{d}y-\Lambda_0(t)\right|\\
&\leq
\sup_{x\in[0,\tau^*+b]}\left|\Lambda_n(t)-\Lambda_0(t)\right|\int_{-1}^1 |k(y)|\,\mathrm{d}y
+
\left|
\int_{-1}^{t/b\wedge 1} k(y)\Lambda_0(t-by)\,\mathrm{d}y-\Lambda_0(t)\right|.
\end{split}
\]
Since there exists $M<\tau_H$ such that, for sufficiently large $n$, $\tau^*+b<M$,
according to~\eqref{eqn:Breslow}, the first term on the right hand side is of the order $O_p(n^{-1/2})$.
For the second term we distinguish between $t\geq b$ and $t<b$.
When $t\geq b$, then with~\eqref{def:kernel},
\[
\begin{split}
\left|
\int_{-1}^{t/b} k(y)\Lambda_0(t-by)\,\mathrm{d}y-\Lambda_0(t)
\right|
&\leq
\int_{-1}^1
|k(y)|
\Big|
\Lambda_0(t-by)-\Lambda_0(t)
\Big|\,\mathrm{d}y\\
&=
b
\sup_{t\in[0,\tau_H]}|\lambda_0(t)|
\int_{-1}^1
|k(y)|
\,\mathrm{d}y
\to0,
\end{split}
\]
uniformly for $t\in[b,\tau^*]$.
When $t<b$, then again with~\eqref{def:kernel}, we can write
\[
\begin{split}
\left|
\int_{-1}^{t/b} k(y)\Lambda_0(t-by)\,\mathrm{d}y-\Lambda_0(t)
\right|
&\leq
\int_{-1}^{t/b} |k(y)|
\Big|
\Lambda_0(t-by)-\Lambda_0(t)
\Big|\,\mathrm{d}y
+
\Lambda_0(t)
\int_{t/b}^1 |k(y)|\,\mathrm{d}y\\
&\leq
O(b)+b\lambda_0(b)
\int_{-1}^1 |k(y)|\,\mathrm{d}y
\to0,
\end{split}
\]
uniformly for $t\in[0,b]$.
It follows that
\begin{equation}
\label{eq:Lemma54 GS}
\sup_{t\in[0,\tau^*]}
\left|\hat\Lambda_n^s(t)-\Lambda_0(t)\right|=o_P(1),
\end{equation}
which proves condition~\eqref{cond d} of Lemma~\ref{lem:monotone2}.
The result now follows from Lemma~\ref{lem:monotone2}.
\end{proof}

\begin{proof}[Proof of Theorem~\ref{theo:as.distrGS}]
From~\eqref{eq:Lemma53 GS}, similar to~\eqref{eq:taylor h},
we find
\[
\begin{split}
&
\tilde{\lambda}_n^{\mathrm{naive}}(x)-\lambda_0(x)\\
&=
\int_{-1}^1 k(y)
\left\{
\lambda_0(x-by)-\lambda_0(x)
\right\}\,\mathrm{d}y
+
b^{-1}
\int_{-1}^1
\left(\Lambda_n(x-by)-\Lambda_0(x-by)\right)
k'(y)\,\mathrm{d}y\\
&=
\frac{(-b)^m}{m!}
\int_{-1}^1
\lambda_0^{(m)}(\xi_{xy})k(y)y^m\,\mathrm{d}y
+
b^{-1}
\int_{-1}^1
\left(\Lambda_n(x-by)-\Lambda_0(x-by)\right)
k'(y)\,\mathrm{d}y,
\end{split}
\]
for some $|\xi_{xy}-x|\leq |by|$.
It follows that
\[
\begin{split}
&
\sup_{x\in[\ell,M]}
\left|
\tilde{\lambda}_n^{\mathrm{naive}}(x)-\lambda_0(x)
\right|\\
&\leq
\frac{b^m}{m!}
\sup_{t\in[0,\tau_H]}\left|\lambda_0^{(m)}(t)\right|
\int_{-1}^1
|k(y)||y|^m\,\mathrm{d}y
+
b^{-1}
\sup_{x\in[\ell,M]}
\left|
\Lambda_n(x)-\Lambda_0(x)
\right|
\int_{-1}^1
|k'(y)|\,\mathrm{d}y
=
o_p(1).
\end{split}
\]
Similar to~\eqref{eq:h term}, the first term on the right hand side is of the order $O(b^m)$,
and according to~\eqref{eqn:Breslow} the second term is of the order $O_p(b^{-1}n^{-1/2})$.
The first statement now follows directly from Corollary~\ref{cor:ISBE=naive}.

To obtain the asymptotic distribution, note that from~\eqref{eq:decomp naive GS}, \eqref{eqn:lambda0} and~\eqref{eqn:Breslow}, we have
\begin{equation}
\label{eqn:asymptotic-mean2}
\begin{split}
&n^{m/(2m+1)}
\left(
\tilde{\lambda}_n^{\mathrm{naive}}(x)-\lambda_0(x)
\right)\\
&=
n^{m/(2m+1)}
\left(
\int k_b(x-u)\,\lambda_0(u)\,\mathrm{d}u-\lambda_0(x)
\right)\\
&\quad+
n^{m/(2m+1)}
\int k_b(x-u)\frac{\delta}{\Phi(u;\beta_0)}\mathrm{d}(\p_n-\p)(u,\delta,z)\\
&\qquad+
n^{-(m+1)/(2m+1)}
\sum_{i=1}^n k_b(x-T_i)\Delta_i
\left(
\frac{1}{\Phi_n(T_i;\hat{\beta}_n)}-\frac{1}{\Phi(T_i;\beta_0)}
\right).
\end{split}
\end{equation}
We find that,
the first term in the right hand side of~\eqref{eqn:asymptotic-mean2} converges to $\mu$,
since
\begin{equation}
\label{eqn:asymptotic_mean}
\begin{split}
&
n^{m/(2m+1)}
\left(
\int k_b(x-u)\,\lambda_0(u)\,\mathrm{d}u-\lambda_0(x)
\right)\\
&=
n^{m/(2m+1)}
\int_{-1}^1 k(y)
\left\{
\lambda_0(x-by)-\lambda_0(x)
\right\}\,\mathrm{d}y\\
&=
n^{m/(2m+1)}
\frac{(-b)^m}{m!}
\int_{-1}^1
\lambda_0^{(m)}(\xi_{xy})k(y)y^m\,\mathrm{d}y
\to
\frac{(-c)^m}{m!}
\lambda_0^{(m)}(x)
\int_{-1}^1 k(y)y^m\,\mathrm{d}y,
\end{split}
\end{equation}
for some $|\xi_{xy}-x|\leq |by|$.
Let $0<M<M'<\tau_H$, so that $x+b\leq M'$ for sufficiently large $n$.
Because $1/\Phi_n(M';\hat{\beta}_n)=O_p(1)$, similar to~\eqref{eq:Phihat-Phi0}
\[
\sup_{u\in[0,M']}
\left|
\frac{1}{\Phi_n(u;\hat{\beta}_n)}-\frac{1}{\Phi(u;\beta_0)}
\right|
\leq
\sup_{u\in[0,M']}
\left|
\frac{\Phi_n(u;\hat{\beta}_n)-\Phi(u;\beta_0)}{\Phi_n(M';\hat{\beta}_n)\Phi(M';\beta_0)}
\right|
=
O_p(n^{-1/2}),
\]
and similar to~\eqref{eq:EYn2^2}
\[
\text{Var}
\left(
n^{-(m+1)/(2m+1)}
\sum_{i=1}^n
\left|
k_b(x-T_i)
\right|\Delta_i\right)
=
O(n^{-1}),
\]
so that the last term on the right hand side of~\eqref{eqn:asymptotic-mean2} converges to zero in probability.
The second term on the right hand side of~\eqref{eqn:asymptotic-mean2} can be written as
\[
\sum_{i=1}^n
\left(
Y_{ni}-\E\left[Y_{ni}\right]
\right),
\quad
Y_{ni}
=
n^{-(m+1)/(2m+1)}
k_b(x-T_i)\frac{\Delta_i}{\Phi(T_i;\beta_0)}.
\]
where similar to~\eqref{eq:EYn2^2},
\begin{equation}
\label{eq:EYn2^2 GS}
\begin{split}
\text{Var}(Y_{ni})
&=
\E\left[Y_{ni}^2\right]+O(n^{-2(m+1)/(2m+1)})\\
&=
n^{-2(m+1)/(2m+1)}b^{-1}
\int_{-1}^1 \frac{k^2(y)h(x-by)}{\Phi^2(x-by;\beta_0)}\,\mathrm{d}y
+
O\left(n^{-2(m+1)/(2m+1)}\right)\\
&=
n^{-1}\sigma^2+o(n^{-1}).
\end{split}
\end{equation}
Moreover,
\[
\sum_{i=1}^n
\E\left[|Y_{ni}|^2\1_{\{|Y_{ni}|>\epsilon \}} \right]
\leq
n^{-2(m+1)/(2m+1)}b^{-2}
\sup_{y\in[-1,1]}|k(y)|
\sum_{i=1}^n\p\left(|Y_{ni}|>\epsilon\right),
\]
where the right hand side tends to zero, because with~\eqref{eq:EYn2^2 GS},
\[
\sum_{i=1}^n
\p\left(|Y_{ni}|>\epsilon\right)
\leq
\sum_{i=1}^n
\frac{\E|Y_{ni}|^2}{\epsilon^2}
=
O(1).
\]By Lindeberg-Feller central limit theorem, we obtain
\[
\sum_{i=1}^n \left(Y_{ni}-\E\left[Y_{ni}\right]\right)\xrightarrow{d}N(0,\sigma^2),
\]
which determines the asymptotic distribution of $\tilde{\lambda}_n^{\mathrm{naive}}(x)$.
Then, by Corollary~\ref{cor:ISBE=naive}, this also gives the asymptotic distribution of $\tilde{\lambda}^{GS}_n(x)$.
\end{proof}

\begin{proof}[Proof of Theorem~\ref{theo:ISBE asymptotic equivalence}]
We write
\[
\begin{split}
&
n^{m/(2m+1)}
\left(
\tilde{\lambda}_n^{SG}(x)-\tilde{\lambda}_n^{GS}(x)
\right)\\
&=
n^{m/(2m+1)}
\left(
\tilde{\lambda}_n^{SG}(x)-\tilde{\lambda}_n^{naive}(x)
\right)
+
n^{m/(2m+1)}
\left(
\tilde{\lambda}_n^{naive}(x)-\tilde{\lambda}_n^{GS}(x)
\right).
\end{split}
\]
By Corollary~\ref{cor:ISBE=naive}, the second term on the right hand side converges to zero in probability.
Furthermore, as can be seen from the proof of Theorem~3.5 in~\cite{LopuhaaMustaSI2016},
\[
\begin{split}
&
n^{m/(2m+1)}
\tilde{\lambda}^{SG}_n(x)\\
&=
n^{m/(2m+1)}
\int k_b(x-u)\,\mathrm{d}\Lambda_0(u)
+
n^{m/(2m+1)}
\int
\frac{\delta k_b(x-u)}{\Phi(u;\beta_0)}\,\mathrm{d}(\p_n-\p)(u,\delta,z)+o_p(1).
\end{split}
\]
Similarly, from the proof of Theorem~\ref{theo:as.distrGS}, we have
\[
\begin{split}
&
n^{m/(2m+1)}
\tilde{\lambda}_n^{naive}(x)
=
n^{m/(2m+1)}
\int k_b'(x-u)\Lambda_n(u)\,\mathrm{d}u\\
&=
n^{m/(2m+1)}
\int k_b(x-u)\,\mathrm{d}\Lambda_n(u)\\
&=
n^{m/(2m+1)}
\int k_b(x-u)\,\mathrm{d}\Lambda_0(u)
+
n^{m/(2m+1)}
\int k_b(x-u)\,\mathrm{d}(\Lambda_n(u)-\Lambda_0(u))\\
&=
n^{m/(2m+1)}
\int k_b(x-u)\,\mathrm{d}\Lambda_0(u)
+
n^{m/(2m+1)}
\int
\frac{\delta k_b(x-u)}{\Phi(u;\beta_0)}\,\mathrm{d}(\p_n-\p)(u,\delta,z)+o_p(1).
\end{split}
\]
From this it immediately follows that
\[
n^{m/(2m+1)}
\left(
\tilde{\lambda}_n^{SG}(x)-\tilde{\lambda}_n^{naive}(x)
\right)
=
o_p(1),
\]
and hence, with Corollary~\ref{cor:ISBE=naive}, also
\[
n^{m/(2m+1)}
\left(
\tilde{\lambda}_n^{SG}(x)-\tilde{\lambda}_n^{GS}(x)
\right)
=
o_p(1).
\]
The second statement about $\hat{\lambda}_n^{SM}(x)$, is immediate using the asymptotic equivalence in~\eqref{eq:asymp norm SM SG}.
\end{proof}

\paragraph{\textbf{Acknowledgement.}}
We like to thank two anonymous referees for their comments and suggestions that helped to improve the lay out and content of the paper.

\section*{References}
\bibliography{shapeconstrained-estimation}

\def\cprime{$'$}
\begin{thebibliography}{38}
\expandafter\ifx\csname natexlab\endcsname\relax\def\natexlab#1{#1}\fi
\expandafter\ifx\csname url\endcsname\relax
  \def\url#1{\texttt{#1}}\fi
\expandafter\ifx\csname urlprefix\endcsname\relax\def\urlprefix{URL }\fi

\bibitem[{Burr(1994)}]{Burr94}
Burr, D., 1994. A comparison of certain bootstrap confidence intervals in the
  {C}ox model. J. Amer. Statist. Assoc. 89~(428), 1290--1302.

\bibitem[{Cheng and Lin(1981)}]{chenglin1981}
Cheng, K.~F., Lin, P.~E., 1981. Nonparametric estimation of a regression
  function. Z. Wahrsch. Verw. Gebiete 57~(2), 223--233.
\newline\urlprefix\url{http://dx.doi.org/10.1007/BF00535491}

\bibitem[{Cheng et~al.(2006)Cheng, Hall, and Tu}]{CHT06}
Cheng, M.-Y., Hall, P., Tu, D., 2006. Confidence bands for hazard rates under
  random censorship. Biometrika 93~(2), 357--366.

\bibitem[{Chung and Chang(1994)}]{CC94}
Chung, D., Chang, M.~N., 1994. An isotonic estimator of the baseline hazard
  function in {C}ox's regression model under order restriction. Statist.
  Probab. Lett. 21~(3), 223--228.

\bibitem[{Cox(1972)}]{Cox72}
Cox, D.~R., 1972. Regression models and life-tables. J. Roy. Statist. Soc. Ser.
  B 34, 187--220, with discussion by F. Downton, Richard Peto, D. J.
  Bartholomew, D. V. Lindley, P. W. Glassborow, D. E. Barton, Susannah Howard,
  B. Benjamin, John J. Gart, L. D. Meshalkin, A. R. Kagan, M. Zelen, R. E.
  Barlow, Jack Kalbfleisch, R. L. Prentice and Norman Breslow, and a reply by
  D. R. Cox.

\bibitem[{Cox(1975)}]{Cox75}
Cox, D.~R., 1975. Partial likelihood. Biometrika 62~(2), 269--276.

\bibitem[{Cox and Oakes(1984)}]{CO84}
Cox, D.~R., Oakes, D., 1984. Analysis of survival data. Monographs on
  Statistics and Applied Probability. Chapman \& Hall, London.

\bibitem[{Durot(2007)}]{durot2007}
Durot, C., 2007. On the {$\Bbb L_p$}-error of monotonicity constrained
  estimators. Ann. Statist. 35~(3), 1080--1104.

\bibitem[{Durot et~al.(2013)Durot, Groeneboom, and Lopuha{\"a}}]{DGL13}
Durot, C., Groeneboom, P., Lopuha{\"a}, H.~P., 2013. Testing equality of
  functions under monotonicity constraints. J. Nonparametr. Stat. 25~(4),
  939--970.

\bibitem[{Efron(1977)}]{Efron72}
Efron, B., 1977. The efficiency of {C}ox's likelihood function for censored
  data. J. Amer. Statist. Assoc. 72~(359), 557--565.

\bibitem[{Eggermont and LaRiccia(2000)}]{eggermont-lariccia2000}
Eggermont, P. P.~B., LaRiccia, V.~N., 2000. Maximum likelihood estimation of
  smooth monotone and unimodal densities. Ann. Statist. 28~(3), 922--947.
\newline\urlprefix\url{http://dx.doi.org/10.1214/aos/1015952005}

\bibitem[{Friedman and Tibshirani(1984)}]{friedmantibshirani1984}
Friedman, J., Tibshirani, R., 1984. The monotone smoothing of scatter plots.
  Technometrics 26, 243--350.

\bibitem[{Gonz{\'a}lez-Manteiga et~al.(1996)Gonz{\'a}lez-Manteiga, Cao, and
  Marron}]{GCM96}
Gonz{\'a}lez-Manteiga, W., Cao, R., Marron, J.~S., 1996. Bootstrap selection of
  the smoothing parameter in nonparametric hazard rate estimation. J. Amer.
  Statist. Assoc. 91~(435), 1130--1140.

\bibitem[{Grenander(1956)}]{grenander1956}
Grenander, U., 1956. On the theory of mortality measurement. {II}. Skand.
  Aktuarietidskr. 39, 125--153 (1957).

\bibitem[{Groeneboom and Jongbloed(2010)}]{GJ10}
Groeneboom, P., Jongbloed, G., 2010. Generalized continuous isotonic
  regression. Statist. Probab. Lett. 80~(3-4), 248--253.

\bibitem[{Groeneboom and Jongbloed(2013)}]{GJ13}
Groeneboom, P., Jongbloed, G., 2013. Smooth and non-smooth estimates of a
  monotone hazard. In: From probability to statistics and back:
  high-dimensional models and processes. Vol.~9 of Inst. Math. Stat. (IMS)
  Collect. Inst. Math. Statist., Beachwood, OH, pp. 174--196.

\bibitem[{Groeneboom and Jongbloed(2015)}]{GJ15}
Groeneboom, P., Jongbloed, G., 2015. Nonparametric confidence intervals for
  monotone functions. Ann. Statist. 43~(5), 2019--2054.

\bibitem[{Groeneboom et~al.(2010)Groeneboom, Jongbloed, and Witte}]{GJW10}
Groeneboom, P., Jongbloed, G., Witte, B.~I., 2010. Maximum smoothed likelihood
  estimation and smoothed maximum likelihood estimation in the current status
  model. Ann. Statist. 38~(1), 352--387.

\bibitem[{Hall(1992)}]{Hall92}
Hall, P., 1992. Effect of bias estimation on coverage accuracy of bootstrap
  confidence intervals for a probability density. Ann. Statist. 20~(2),
  675--694.

\bibitem[{{Lopuha{\"a}} and {Musta}(2015)}]{LM15}
{Lopuha{\"a}}, H.~P., {Musta}, E., 2015. {Smooth estimation of a monotone
  hazard and a monotone density under random censoring}. \textrm{Accepted by}
  Statistica Neerlandica.

\bibitem[{Lopuha\"a and Musta(2016)}]{LopuhaaMustaSI2016}
Lopuha\"a, H.~P., Musta, E., 2016. Smoothed isotonized estimators of a monotone
  baseline hazard in the cox model. Submitted.

\bibitem[{Lopuha{\"a} and Nane(2013)}]{LopuhaaNane2013}
Lopuha{\"a}, H.~P., Nane, G.~F., 2013. Shape constrained non-parametric
  estimators of the baseline distribution in {C}ox proportional hazards model.
  Scand. J. Stat. 40~(3), 619--646.

\bibitem[{Mammen(1991)}]{mammen1991}
Mammen, E., 1991. Estimating a smooth monotone regression function. Ann.
  Statist. 19~(2), 724--740.

\bibitem[{Mukerjee(1988)}]{mukerjee1988}
Mukerjee, H., 1988. Monotone nonparameteric regression. Ann. Statist. 16~(2),
  741--750.
\newline\urlprefix\url{http://dx.doi.org/10.1214/aos/1176350832}

\bibitem[{M\"uller and Wang(1990{\natexlab{a}})}]{mullerwang1990}
M\"uller, H.-G., Wang, J.-L., 1990{\natexlab{a}}. Locally adaptive hazard
  smoothing. Probab. Theory Related Fields 85~(4), 523--538.
\newline\urlprefix\url{http://dx.doi.org/10.1007/BF01203169}

\bibitem[{M\"uller and Wang(1990{\natexlab{b}})}]{MW90}
M\"uller, H.-G., Wang, J.-L., 1990{\natexlab{b}}. Locally adaptive hazard
  smoothing. Probab. Theory Related Fields 85~(4), 523--538.
\newline\urlprefix\url{http://dx.doi.org/10.1007/BF01203169}

\bibitem[{Nane(2013)}]{Nane}
Nane, G.~F., 2013. Shape Constrained Nonparametric Estimation in the Cox Model.
  Delft University of Technology.

\bibitem[{Ramlau-Hansen(1983)}]{ramlau-hansen1983}
Ramlau-Hansen, H., 1983. Smoothing counting process intensities by means of
  kernel functions. Ann. Statist. 11~(2), 453--466.
\newline\urlprefix\url{http://dx.doi.org/10.1214/aos/1176346152}

\bibitem[{Ramsay(1998)}]{ramsay1998}
Ramsay, J.~O., 1998. Estimating smooth monotone functions. J. R. Stat. Soc.
  Ser. B Stat. Methodol. 60~(2), 365--375.
\newline\urlprefix\url{http://dx.doi.org/10.1111/1467-9868.00130}

\bibitem[{Ren and Zhou(2011)}]{RZ11}
Ren, J.-J., Zhou, M., 2011. Full likelihood inferences in the {C}ox model: an
  empirical likelihood approach. Ann. Inst. Statist. Math. 63~(5), 1005--1018.
\newline\urlprefix\url{http://dx.doi.org/10.1007/s10463-010-0272-y}

\bibitem[{Tanner and Wong(1983)}]{tanner-wong1983}
Tanner, M.~A., Wong, W.~H., 1983. The estimation of the hazard function from
  randomly censored data by the kernel method. Ann. Statist. 11~(3), 989--993.
\newline\urlprefix\url{http://dx.doi.org/10.1214/aos/1176346265}

\bibitem[{Tsiatis(1981)}]{Tsiatis81}
Tsiatis, A.~A., 1981. A large sample study of {C}ox's regression model. Ann.
  Statist. 9~(1), 93--108.

\bibitem[{van~der Vaart and van~der Laan(2003)}]{vdvaart-vdlaan2003}
van~der Vaart, A.~W., van~der Laan, M.~J., 2003. Smooth estimation of a
  monotone density. Statistics 37~(3), 189--203.
\newline\urlprefix\url{http://dx.doi.org/10.1080/0233188031000124392}

\bibitem[{van~der Vaart and Wellner(1996)}]{VW96}
van~der Vaart, A.~W., Wellner, J.~A., 1996. Weak convergence and empirical
  processes. Springer Series in Statistics. Springer-Verlag, New York, with
  applications to statistics.

\bibitem[{van Geloven et~al.(2013)van Geloven, Martin, Damman, de~Winter,
  Tijssen, and Lopuha{\"a}}]{Geloven13}
van Geloven, N., Martin, I., Damman, P., de~Winter, R.~J., Tijssen, J.~G.,
  Lopuha{\"a}, H.~P., 2013. Estimation of a decreasing hazard of patients with
  acute coronary syndrome. Stat. Med. 32~(7), 1223--1238.

\bibitem[{Wells(1994)}]{wells1994}
Wells, M.~T., 1994. Nonparametric kernel estimation in counting processes with
  explanatory variables. Biometrika 81~(4), 795--801.
\newline\urlprefix\url{http://dx.doi.org/10.1093/biomet/81.4.795}

\bibitem[{Wright(1982)}]{wright1982}
Wright, F.~T., 1982. Monotone regression estimates for grouped observations.
  Ann. Statist. 10~(1), 278--286.
\newline\urlprefix\url{http://links.jstor.org/sici?sici=0090-5364(198203)10:1<278:MREFGO>2.0.CO;2-E&origin=MSN}

\bibitem[{Xu et~al.(2014)Xu, Sen, and Ying}]{XSY2014}
Xu, G., Sen, B., Ying, Z., 2014. Bootstrapping a change-point {C}ox model for
  survival data. Electron. J. Stat. 8~(1), 1345--1379.

\end{thebibliography}

\end{document}